\documentclass{article}
\usepackage{hyperref}
\usepackage{graphicx}
\usepackage[all]{xy} 
\usepackage{amsmath}
\usepackage{amsthm}
\usepackage{array}  
\usepackage{amssymb}
\usepackage{calc}
\usepackage{supertabular}
\usepackage{longtable}
\usepackage{txfonts}
\usepackage{textcomp}
\usepackage{colortbl}
\usepackage{tabularx} 
\usepackage[table]{xcolor}
\usepackage{lscape}
\usepackage{rotating}
\usepackage{tensor}
\usepackage{arydshln} 
\usepackage{bbm} 
\usepackage{appendix} 
\usepackage{placeins} 
\usepackage[T1]{fontenc} 
\usepackage{titletoc}
\usepackage{pgf}
\usepackage{tikz}
\usetikzlibrary{matrix,arrows}
\usetikzlibrary{patterns}
\usetikzlibrary{shapes,decorations,shadows}
\usetikzlibrary{decorations.pathmorphing}
\usetikzlibrary{decorations.shapes}

\DeclareMathOperator{\Cone}{Cone} 
\DeclareMathOperator{\add}{add}  
\DeclareMathOperator{\Mod}{Mod} 
 
\DeclareMathOperator{\Hom}{Hom} 
\DeclareMathOperator{\rad}{rad} 
\DeclareMathOperator{\inj}{inj} 
\DeclareMathOperator{\tor}{tor}

\DeclareMathOperator{\Proj}{Proj} 
\DeclareMathOperator{\im}{im} 
\DeclareMathOperator{\id}{id}

\DeclareMathOperator{\rep}{rep}

\DeclareMathOperator{\GL}{GL}

\DeclareMathOperator{\codim}{codim} 
\DeclareMathOperator{\Spec}{Spec} 
 
\DeclareMathOperator{\sst}{sst}

\DeclareMathOperator{\hb}{hb} 
\DeclareMathOperator{\hd}{hd} 
\DeclareMathOperator{\hc}{hc} 
\DeclareMathOperator{\hs}{hs} 
\DeclareMathOperator{\ch}{ch} 
\DeclareMathOperator{\vc}{vc} 
\DeclareMathOperator{\vs}{vs} 
\DeclareMathOperator{\cv}{cv} 
\DeclareMathOperator{\vb}{vb} 
\DeclareMathOperator{\vd}{vd} 

\definecolor{darkred}{rgb}{0.5,0,0}
\definecolor{llllllllightblue}{rgb}{0.95,0.95,1.2}
\definecolor{lllllllightblue}{rgb}{0.9,0.9,1}
\definecolor{llllllightblue}{rgb}{0.85,0.85,0.95}
\definecolor{lllllightblue}{rgb}{0.8,0.8,0.9}
\definecolor{llllightblue}{rgb}{0.75,0.75,0.85}
\definecolor{lllightblue}{rgb}{0.7,0.7,0.8}
\definecolor{llightblue}{rgb}{0.65,0.65,0.75}
\definecolor{lightblue}{rgb}{0.6,0.6,0.7}
\definecolor{ightblue}{rgb}{0.55,0.55,0.65}
\definecolor{ghtblue}{rgb}{0.5,0.5,0.6}
\definecolor{htblue}{rgb}{0.45,0.45,0.55}
\definecolor{tblue}{rgb}{0.4,0.4,0.5}
\definecolor{blue}{rgb}{0.35,0.35,0.45}
\definecolor{lue}{rgb}{0.3,0.3,0.4}
\definecolor{lblue}{rgb}{0.3,0.0,4.4}
\definecolor{darkblue}{rgb}{0,0.0,0.5}
\definecolor{lightgrey}{rgb}{0.8,0.8,0.8}
\newcommand{\Det}{{\det}}
\newcommand*{\punkte}{\dots\unkern}

\newcommand{\Fa}{\mathcal{F}} 
\newcommand{\Pa}{\mathcal{P}} 
\newcommand{\Ha}{\mathcal{H}} 
 
\newcommand{\Q}{\mathcal{Q}} 
\newcommand{\Orb}{\mathcal{O}} 
\newcommand{\N}{\mathcal{N}} 
 
\newcommand{\quot}{/\!\!/}

\newcommand{\dimv}{\underline{\dim}}
\newcommand{\dfp}{\underline{d}_{P}} 
\newcommand{\df}{\underline{d}} 
\newcommand{\dfs}{\underline{d}_B}

\newtheorem{theorem}{Theorem}[section]
\newtheorem{lemma}[theorem]{Lemma}
\newtheorem{definition}[theorem]{Definition}
\newtheorem{proposition}[theorem]{Proposition}
\newtheorem{corollary}[theorem]{Corollary}

\newtheorem{example}[theorem]{Example}

\begin{document}
\parindent0pt
\title{\bf Non-reductive conjugation on the nilpotent cone}

\author{Magdalena Boos\\ Fachbereich C - Mathematik\\ Bergische Universit\"at Wuppertal\\ D - 42097 Wuppertal\\ boos@math.uni-wuppertal.de}
\date{}
\maketitle

\begin{abstract}
We consider the conjugation-action of an arbitrary upper-block parabolic subgroup of $\GL_n(\mathbf{C})$, especially of the Borel subgroup $B$ and of the standard unipotent subgroup $U$ of the latter on the nilpotent cone of complex nilpotent matrices.  We obtain generic normal forms of the orbits and describe generating (semi-) invariants for the Borel semi-invariant ring as well as for the $U$-invariant ring. The latter is described in more detail in terms of algebraic quotients by a special toric variety closely related. The study of a GIT-quotient for the Borel-action is initiated.
\end{abstract}
\section{Introduction}\label{intro}
The "horizontal" study of algebraic group actions on affine varieties by parametric families of orbits and quotients are a natural topic in algebraic Lie theory.\\[1ex]
In particular, the study of the adjoint action of a reductive algebraic group on its Lie algebra and numerous variants thereof yield various examples. One of these is the study  of complex (nilpotent) square matrices up to isomorphism.\\[1ex]
 Algebraic group actions of reductive groups have particularly been discussed elaborately in connection with orbit spaces and more generally algebraic quotients, even though their application to concrete examples is far from being trivial. In case of a non-reductive group, even most of these results fail to hold true immediately.\\[1ex]
For example, Hilbert's Theorem \cite{Hi} yields that for reductive groups, the invariant ring is finitely generated; and a criterion for algebraic quotients is valid \cite{Kr}. In 1958, though, M. Nagata \cite{Na1} constructed a counterexample of a not finitely generated invariant ring corresponding to a non-reductive algebraic group action, which answered Hilbert's fourteenth problem in the negative.\\[1ex]
The corresponding invariant rings of algebraic actions of unipotent subgroups that are induced by reductive groups are always finitely generated \cite{Kr}, though.\\[1ex]
 We turn our main attention towards algebraic non-reductive group actions that are induced by the conjugation action of the general linear group $\GL_n$ over $\mathbf{C}$. For example, the standard parabolic subgroups $P$ (and, therefore, the standard Borel subgroup $B$) and the standard unipotent subgroup $U$ of $\GL_n$ are not reductive. 
It suggests itself to consider their action on the variety $\N$ of complex nilpotent matrices of square size $n$, also known as the \textit{nilpotent cone}, via conjugation which we discuss in this work.\\[1ex]
We begin by providing a short introduction of the theoretical background in Section \ref{theor}.\\[1ex]
In Section \ref{transsect}, an associated fibre bundle is proved, which yields a translation of the classification problem of the $P$-orbits in $\N$ to  the description of certain isomorphism classes of representations of a finite-dimensional algebra. The translation will be used later on to study  (algebraic) quotiens of the above mentioned group actions.\\[1ex]
In \cite{Hal,BoRe}, a generic $B$-normal form on $\N$ is introduced which we generalize to arbitrary upper-block parabolic subgroups in Section \ref{gnfsect}. The generalization is quite natural and extends the before mentioned result.\\[1ex]
In Section \ref{generation}, we describe $B$-semi-invariants and prove that these, in fact, generate the ring of all $B$-semi-invariants. As a direct consequence, we are able to find $U$-invariants that span the $U$-invariant ring.\\[1ex]
 The latter will be made use of to discuss the $U$-invariant ring in more detail in Section \ref{Uquot} by proving a quotient criterion and discussing a toric variety closely related to the algebraic quotient of $\N$ by $U$.\\[1ex]
Finally, we initiate the study of GIT-quotients for the Borel-action in Section \ref{Borelquot}.\\[1ex] 
The cases $n=2$ and $n=3$ are discussed in all detail, that is, the $B$-semi-invariant ring and it's quotient as well as the $U$-invariant ring and its quotient are written down explicitly in Sections \ref{Uquot} and \ref{Borelquot}.\\[1ex]
The results stated in this article represent a part of the outcome of the dissertation \cite{B1}.\\[1ex]
{\bf Acknowledgments:} The author would like to thank M. Reineke for various valuable discussions concerning the methods and results of this work. Furthermore,  A. Melnikov, K. Bongartz and R. Tange are being thanked for inspirational thoughts and helpful remarks.\\[1ex]
The published version of this article is \cite{B3}.
\section{Theoretical background}\label{theor}
Let us denote by $K\coloneqq \mathbf{C}$ the field of complex numbers and by $\GL_n\coloneqq\GL_n(K)$ the general linear group for a fixed integer $n\in\textbf{N}$ regarded as an affine variety. 

\subsection{(Semi-) Invariants and quotients}\label{generic}
We start by providing basic knowledge about \mbox{(semi-)} invariants and quotients \cite{Kr,Mu}. Let $G$ be a linear algebraic group and let $X$ be an affine $G$-variety. We denote by  $X(G)$ the \textit{character group} of $G$; a global section $f\in K[X]$ is called a \textit{$G$-semi-invariant of weight $\chi\in X(G)$} if  $f(g.x)=\chi(g)\cdot f(x)$ for all $x\in X$ and $g\in G$.  \\[1ex]
Let us denote the $\chi$-semi-invariant ring by
\[K[X]^{G}_{\chi}\coloneqq \bigoplus_{n\geq 0}K[X]^{G, n\chi},\]
which is a subring of $K[X]$ and naturally $\textbf{N}$-graded by the sets $K[X]^{G, n\chi}$, that is, by the semi-invariants of weight $n\chi$ (and of degree $n$). The semi-invariant ring corresponding to all characters is denoted by
\[K[X]^{G}_{*}\coloneqq \bigoplus_{\chi\in X(G)}K[X]^{G}_{\chi}.\]
A global section $f\in K[X]$ is called a  \textit{$G$-invariant} if $f(g.x)=f(x)$ for all $x\in X$ and  $g\in G$; the corresponding $G$-invariant ring is denoted by $K[X]^G$.
If the group $G$ is \textit{reductive}, that is, if every linear representation of $G$ can be decomposed into a direct sum of irreducible representations, D. Hilbert showed  that the invariant ring is finitely generated (see \cite{Hi}), even though it can be a problem of large difficulty to find generating invariants.\\[1ex] 
Let $X'$  be yet another affine $G$-variety and let $Y$ be an affine variety.\\[1ex]
A $G$-invariant morphism $\pi\colon X\rightarrow Y=:X\quot G$ is called an \textit{algebraic $G$-quotient of $X$} if 
it fulfills the universal property that for every $G$-invariant morphism $f\colon X\rightarrow Z$, there exists a unique morphism $\hat{f}\colon Y\rightarrow Z$, such that $f=\hat{f}\circ \pi$. If $K[X]^G$ is finitely generated, the variety $\Spec K[X]^G= X\quot G$ induces an algebraic quotient. Each fibre of an algebraic quotient contains exactly one closed orbit, thus, these closed orbits are parametrized.\\[1ex]
In order to calculate an algebraic $G$-quotient of an affine variety, the following criterion (see \cite[II.3.4]{Kr}) can be helpful.
\begin{theorem}\label{criterion}
Let $G$ be a reductive group and let $\pi\colon X\rightarrow Y$ be a $G$-invariant morphism of varieties. If 
\begin{enumerate}
 \item $Y$ is normal,
\item $\codim_Y(\overline{Y\backslash \pi(X)})\geq 2$ (or $\pi$ is surjective if $\dim Y=1$) and
\item on a non-empty open subset $Y_0\subseteq Y$ the fibre $\pi^{-1}(y)$ contains exactly one closed orbit for each $y\in Y_0$,
\end{enumerate}
then $\pi$ is an algebraic $G$-quotient of $X$.
\end{theorem}
In case $G$ is not reductive, there are counterexamples of only infinitely generated invariant rings (see \cite{Na1}). For actions of unipotent subgroups which are induced by reductive group actions, however, the following lemma \cite[III.3.2]{Kr} holds true.
\begin{lemma}\label{Uinvfin}
 Let $U$ be a unipotent subgroup of $G$; the action of $G$ restricts to an action of $U$ on $X$. Then the invariant ring $K[X]^U$ is finitely generated as a $K$-algebra.
\end{lemma}
Given functions $f_0,\punkte,f_s\in K[X]^{G}_{\chi}$, such that all ratios $\frac{f_i}{f_j}$ are $G$-invariant rational functions, the
map  \begin{align}\pi\colon& ~ X ----> \textbf{P}^s \nonumber\\
&~ x~\mapsto~ (f_0(x):\punkte:f_s(x)) \nonumber\end{align}
 is not defined on the common zeros of $f_1,\punkte,f_s$. If we extend the number of functions $f_i$ it is possible that the set of common zeros is diminished even though they in general do not vanish completely.\\[1ex]
 These thoughts suggests the definition of the so-called unstable locus. Let $\chi\in X(G)$ be a $G$-character, then we define the unstable locus of $\chi$ to be the subset of unstable points $x\in X$, that is,  $f(x)=0$ for every $f\in K[X]^{G, n\chi}$ and for every integer $n>0$.\\[1ex] 
We, furthermore, define the semi-stable locus of $\chi$ to be the set of $\chi$-semi-stable points in $X$, that is, of points $x\in X$ for which a $\chi$-semi-invariant $f\in K[X]^{G, n\chi}$ for an integer $n>0$ exists, such that $f(x)\neq 0$.\\[1ex]
We define the so-called GIT-quotient of $X$ by $G$ in direction $\chi$ to be \[X\quot_{\chi}G:= \Proj(K[X]^{G}_{\chi})\] together with the induced morphism
$\pi\colon X^{\chi-\sst}\rightarrow X\quot_{\chi}G$.\\[1ex]
If the linear algebraic group $G$ is reductive, the ring $K[X]^G_{\chi}$ is finitely generated (see \cite[6.1(b)]{Mu} or \cite{Re} for more information on the subject) and a morphism
\begin{align}
\pi\lvert_{\chi}\colon& ~ X^{\chi-\sst} \rightarrow ~~~~~~~~~X\quot_{\chi} G ~~~~~~~~~~\subseteq \Proj K[x_0,\punkte,x_s]\nonumber\\
&~~~~x~~~~\mapsto~ (f_0(x):\punkte:f_s(x)). \nonumber
\end{align}
is obtained, where $f_0,\ldots,f_s\in K[X]^G_{\chi}$ are generating semi-invariants of degrees $a_0,\ldots, a_s$ and $x_i$ is of weight $a_i$ for all $i\in\{0,\punkte,s\}$. We call $\pi\lvert_{\chi}$ a GIT-quotient map  of $X$ by $G$ in direction $\chi$. 
\subsection{Toric varieties}
Since our considerations will involve the notion of a toric variety, we discuss it briefly. For more information on the subject, the reader is referred to \cite{Fu}.\\[1ex]
A \textit{toric variety} is an irreducible variety $X$ which containes $(K^*)^n$ as an open subset, such that the action of $(K^*)^n$ on itself extends to an action of $(K^*)^n$ on $X$.\\[1ex]
Let $N$ be a \textit{lattice}, that is, a free abelian group $N$ of finite rank. By $M:=\Hom_{\textbf{Z}}(N,\textbf{Z})$ we denote the \textit{dual lattice}, together with the induced dual pairing $\langle\_,\_\rangle$.  Consider the vector space $N_{\textbf{R}}:=N\otimes_{\textbf{Z}}\textbf{R}\cong \textbf{R}^n$.\\[1ex]
A subset $\sigma\subseteq N_{\textbf{R}}$ is called a \textit{strongly convex rational polyhedral cone} if $\sigma\cap (-\sigma) =\{0\}$ and if there is a finite set $S\subseteq N$ that generates $\sigma$, that is,
\[\sigma = \Cone(S):= \left\lbrace \sum\limits_{s\in S} \lambda_s\cdot s \mid \lambda_s\geq 0 \right\rbrace. \]
Given a strongly convex rational polyhedral cone $\sigma$, we define its dual by
\[\sigma^{\vee}:=\{m\in \Hom_{\textbf{R}}(\textbf{R}^n,\textbf{R})\mid \langle m,v\rangle\geq 0~\textrm{for~all}~v\in\sigma\}\]
and its corresponding additive semigroup by $S_{\sigma}:=\sigma^{\vee}\cap M$, which is finitely generated due to Gordon's Lemma (see \cite{Fu}). Note that  if $\sigma$ is a maximal dimensional strongly convex rational polyhedral cone, then $\sigma^{\vee}$ is one as well. We associate to it the semigroup algebra $KS_{\sigma}$ and obtain an affine toric variety $\Spec KS_{\sigma}$. The following lemma can be found in \cite{CLS}.
\begin{lemma}\label{toricco}
 An affine toric variety $X$ is isomorphic to $\Spec KS_{\sigma}$ for some strongly convex rational polyhedral cone $\sigma$ if and only if $X$ is normal.
\end{lemma}

 \section{Translation to a representation-theoretic setup}\label{transsect}
We fix an upper-block parabolic subgroup $P$ of $\GL_n$ of block sizes $(b_1,\punkte,b_p)$, the standard Borel subgroup $B\subset\GL_n$ and its unipotent subgroup $U\subset B$ and will discuss their actions on the nilpotent cone $\N$ of nilpotent complex matrices.\\[1ex] 
We start by recapitulating basic knowledge about the representation theory of finite-dimensional algebras before translating the above setup into this context.\\[1ex] 
A \textit{finite quiver} $\Q$ is a directed graph $\Q=(\Q_0,\Q_1,s,t)$ with a finite set of \textit{vertices} $\Q_0$ and  a finite set of \textit{arrows} $\Q_1$, whose elements are written as $\alpha\colon s(\alpha)\rightarrow t(\alpha)$.
Its \textit{path algebra} $K\Q$ is defined as the $K$-vector space with a basis consisting of all paths in $\Q$, that is, sequences of arrows $\omega=\alpha_s\punkte\alpha_1$, such that $t(\alpha_{k})=s(\alpha_{k+1})$ for all $k\in\{1,\punkte,s-1\}$; we formally include a path $\varepsilon_i$ of length zero for each $i\in \Q_0$ starting and ending in $i$. The multiplication is defined by
\begin{center}
 $\omega\cdot\omega'=\left\{\begin{array}{ll}\omega\omega',&~\textrm{if}~t(\beta_t)=s(\alpha_1);\\
0,&~\textrm{otherwise.}\end{array}\right.$\end{center}
where $\omega\omega'$ is the  concatenation of paths $\omega$ and $\omega'$.\\[1ex]
We define the \textit{radical} $\rad(K\Q)$ of $K\Q$ to be the (two-sided) ideal generated by all paths of positive length; then an arbitrary ideal $I$ of $K\Q$ is called \textit{admissible} if there exists an integer $s$ with $\rad(K\Q)^s\subset I\subset\rad(K\Q)^2$.\\[1ex]
A finite-dimensional \textit{$K$-representation} of $\Q$ is a tuple \[((M_i)_{i\in \Q_0},(M_\alpha\colon M_i\rightarrow M_j)_{(\alpha\colon i\rightarrow j)\in \Q_1}),\] where the $M_i$ are $K$-vector spaces, and the $M_{\alpha}$ are $K$-linear maps.\\[1ex]
 A \textit{morphism of representations} $M=((M_i)_{i\in \Q_0},(M_\alpha)_{\alpha\in \Q_1})$ and
 \mbox{$M'=((M'_i)_{i\in \Q_0},(M'_\alpha)_{\alpha\in \Q_1})$} consists of a tuple of $K$-linear maps $(f_i\colon M_i\rightarrow M'_i)_{i\in \Q_0}$, such that $f_jM_\alpha=M'_\alpha f_i$ for every arrow $\alpha\colon i\rightarrow j$ in $\Q_1$.\\[1ex]
For a representation $M$ and a path $\omega$ in $\Q$ as above, we denote $M_\omega=M_{\alpha_s}\cdot\punkte\cdot M_{\alpha_1}$. A representation $M$ is called \textit{bound by $I$} if $\sum_\omega\lambda_\omega M_\omega=0$ whenever $\sum_\omega\lambda_\omega\omega\in I$.\\[1ex]
We denote by $\rep_K(\Q)$ the abelian $K$-linear category of all representations of $\Q$ and by   $\rep_K(\Q,I)$ the category of representations of $\Q$ bound by $I$; the latter is equivalent to the category of finite-dimensional $K\Q/I$-representations.\\[1ex]
Given a representation $M$ of $\Q$, its \textit{dimension vector} $\dimv M\in\mathbf{N}\Q_0$ is defined by $(\dimv M)_{i}=\dim_K M_i$ for $i\in \Q_0$. Let us fix a dimension vector $\df\in\mathbf{N}\Q_0$, then we denote by $\rep_K(\Q,I)(\df)$ the full subcategory of $\rep_K(\Q,I)$ which consists of representations of dimension vector $\df$.\\[1ex]
By defining the affine space $R_{\df}(\Q):= \bigoplus_{\alpha\colon i\rightarrow j}\Hom_K(K^{d_i},K^{d_j})$, one realizes that its points $m$ naturally correspond to representations $M\in\rep_K(\Q)(\df)$ with $M_i=K^{d_i}$ for $i\in \Q_0$. 
 Via this correspondence, the set of such representations bound by $I$ corresponds to a closed subvariety $R_{\df}(\Q,I)\subset R_{\df}(\Q)$.\\[1ex]
The algebraic group $\GL_{\df}=\prod_{i\in \Q_0}\GL_{d_i}$ acts on $R_{\df}(\Q)$ and on $R_{\df}(\Q,I)$ via base change, furthermore the $\GL_{\df}$-orbits $\Orb_M$ of this action are in bijection to the isomorphism classes of representations $M$ in $\rep_K(\Q,I)(\df)$.
There is an induced $\GL_{\df}$-action on $K[R_{\df}(\Q)]$ which yields the natural notion of semi-invariants. \\[1ex]
Let us denote by $\add\Q$ the \textit{additive category} of $\Q$ with objects $O(i)$ corresponding to the vertices $i\in\Q_0$ and morphisms induced by the paths in $\Q$. Since every representation $M\in \rep_K(\Q)$ can naturally be seen as  a functor from $\add \Q$ to $\Mod K$, we denote this functor by $M$ as well.\\[1ex]
 Let $\phi\colon\bigoplus_{i=1}^n O(i)^{x_i}\rightarrow \bigoplus_{i=1}^n O(i)^{y_i}$ be an arbitrary morphism in $\add \Q$ and consider $\df\in \mathbf{N}\Q_0$, such that $\sum_{i\in \Q_0}x_i\cdot \df_i=\sum_{i\in \Q_0}y_i\cdot \df_i$. An induced so-called determinantal semi-invariant is given by
\[ f_{\phi}\colon~R_{\df}(\Q)\rightarrow K ;~~ m~~\mapsto \det(M(\phi)), \]
where $m\in R_{\df}(\Q)$ and $M\in  \rep_K(\Q)(\df)$ are related via the above mentioned correspondence.
 The following theorem (see \cite{SvB}) is due to A. Schofield and M. van den Bergh.
\begin{theorem}\label{Schosemi}
 The semi-invariants in $K[R_{\df}(\Q)]^{\GL_{\df}}_*$ are spanned by the determinantal semi-invariants $f_{\phi}$.
\end{theorem}
We will make use of the following fact on associated fibre bundles to translate the above described algebraic group action into another algebraic group action in the context of representation theory (see, for example, \cite{Se} or \cite{Bo1}).
\begin{theorem}\label{assocfibrebundles}
Let $G$ be a linear algebraic group, let $X$ and $Y$ be $G-$varieties, and let $\pi\colon X \rightarrow Y$ be a $G$-equivariant morphism. Assume that $Y$ is a single $G$-orbit, $Y = G.y_0$. Let $H$ be the stabilizer of $y_0$ and set $F\coloneqq \pi^{-1} (y_0)$. Then $X$ is isomorphic to the associated fibre bundle $G\times^HF$, and the embedding $\phi\colon F \hookrightarrow X$ induces a bijection $\Phi$ between the $H$-orbits in $F$ and the $G$-orbits in $X$ preserving orbit closures and types of singularities.
\end{theorem}
Let us define  $\Q_p$ to be the quiver
\begin{center}\begin{tikzpicture}
\matrix (m) [matrix of math nodes, row sep=0.01em,
column sep=1.5em, text height=0.5ex, text depth=0.1ex]
{\Q_p\colon & \bullet & \bullet &  \bullet & \cdots  & \bullet & \bullet  & \bullet \\ & \mathrm{1} & \mathrm{2} &  \mathrm{3} & &   \mathrm{p-2} &  \mathrm{p-1}  & \mathrm{p} \\ };
\path[->]
(m-1-2) edge node[above=0.05cm] {$\alpha_1$} (m-1-3)
(m-1-3) edge  node[above=0.05cm] {$\alpha_2$}(m-1-4)
(m-1-6) edge  node[above=0.05cm] {$\alpha_{p-2}$}(m-1-7)
(m-1-7) edge node[above=0.05cm] {$\alpha_{p-1}$} (m-1-8)
(m-1-8) edge [loop right] node{$\alpha$} (m-1-8);\end{tikzpicture}\end{center} 
and consider the finite-dimensional algebra $ K \Q_p/I$, where $I\coloneqq (\alpha^n)$ is an admissible ideal. Let us fix the dimension vector 
\[\dfp\coloneqq(d_1,\punkte,d_p)\coloneqq(b_1,b_1+b_2, \punkte, b_1+...+b_p)\]
 and formally set $b_0=0$. The algebraic group $\GL_{\dfp}$ acts on $R_{\dfp}(\Q_p,I)$; the orbits of this action are in bijection with the isomorphism classes of representations in $\rep_{K}(\Q_p,I)(\dfp)$.\\[1ex]
Let us define $\rep_{K}^{\inj}(\Q_p,I)(\dfp)$ to be the full subcategory of $\rep_{K}(\Q_p,I)(\dfp)$ consisting of representations $((M_i)_{1\leq i\leq p},(M_{\rho})_{\rho\in \Q_1})$, such that $M_{\rho}$ is injective if $\rho=\alpha_i$ for every $i\in\{1,\punkte, p-1\}$. Corresponding to this subcategory, there is an open subset $$R_{\dfp}^{\inj}(\Q_p,I)\subset R_{\dfp}(\Q_p,I),$$ which is stable under the $\GL_{\dfp}$-action.
We denote $\Orb_M:=\GL_{\dfp}.m$ if $m\in R_{\dfp}^{\inj}(\Q_p,I)$ corresponds to the representation $M\in\rep^{\inj}(\Q_p,I)(\dfp)$.\\[1ex]
The following lemma is a slightly different version of \cite[Lemma 3.2]{BoRe}; it can be proved analogously. 
\begin{lemma} \label{bijection}
There is an isomorphism $R_{\dfp}^{\inj}(\Q_p,I)\cong \GL_{\dfp}\times^{P}\N$. Thus, there exists a bijection $\Phi$ between the set of $P$-orbits in $\N$ and the set of $\GL_{\dfp}$-orbits in $R_{\dfp}^{\inj}(\Q_p,I)$, which sends an orbit $P.N\subseteq \N$ to the isomorphism class of the representation
\begin{center}\begin{tikzpicture}
\matrix (m) [matrix of math nodes, row sep=0.05em,
column sep=2em, text height=1.5ex, text depth=0.2ex]
{ K^{d_1} & K^{d_2} & K^{d_3} & \cdots  & K^{d_{p-2}} & K^{d_{p-1}}  & K^{n}\\ };
\path[->]
(m-1-1) edge node[above=0.05cm] {$\epsilon_1$} (m-1-2)
(m-1-2) edge  node[above=0.05cm] {$\epsilon_2$}(m-1-3)
(m-1-3) edge  (m-1-4)
(m-1-4) edge  (m-1-5)
(m-1-5) edge  node[above=0.05cm] {$\epsilon_{p-2}$}(m-1-6)
(m-1-6) edge node[above=0.05cm] {$\epsilon_{p-1}$} (m-1-7)
(m-1-7) edge [loop right] node{$N$} (m-1-7);\end{tikzpicture}\end{center}
 (denoted $M^N$) with natural embeddings $\epsilon_i\colon K^{d_i}\hookrightarrow K^{d_{i+1}}$. This bijection preserves codimensions.
\end{lemma}

\section{Generic normal forms in the nilpotent cone}\label{gnfsect}
We discuss the $P$-action  on the nilpotent cone $\N$ now and introduce a generic normal form. We, thereby, generalize a generic normal form for the orbits of the Borel-action which is introduced in \cite{BoRe,Hal}. 
\begin{definition}
 Let $G$ be an algebraic group acting on an affine Variety $X$. A subset $X_0\subseteq X$ is called a generic normal form, if 
\begin{enumerate}
 \item $G.X_0\subseteq X$ is open and
\item $G.x \neq G.x'$ for all $x,x'\in X_0$, where $x\neq x'$.
\end{enumerate}

\end{definition}

 Let $V$ be an $n$-dimensional $K$-vector space and denote the space of partial $p$-step flags of dimensions $\dfp$ by $\Fa_{\dfp}(V)$, that is, $\Fa_{\dfp}(V)$ contains flags
\[(0=F_0\subset F_1\subset \punkte \subset F_{p-2} \subset F_{p-1} \subset F_{p}=V),\]
such that $\dim_K F_i=d_i$. Let $\varphi$ be a nilpotent endomorphism of $V$ and consider pairs of a nilpotent endomorphism and a $p$-step flag up to base change in $V$, that is, up to the $\GL(V)$-action via $g.(F_*,\varphi) = (gF_*, g\varphi g^{-1})$.\\[1ex]
Let us fix a partial flag $F_*\in \Fa_{\dfp}(V)$ and a nilpotent endomorphism $\varphi$ of $V$.
\begin{lemma}\label{pargenlem}
 The following properties of the pair $(F_*, \varphi)$ are equivalent:
\begin{enumerate}
 \item $\dim_K \varphi^{n-d_k}(F_k)=d_k$ for every $k\in\{0,\punkte, p\}$,
\item there exists a basis $\{w_1,\punkte, w_n\}$ of $V$, such that  for  all    $ k\in\{1,\punkte, p\}$:
\begin{enumerate} 
 \item[($ \textrm{a}_k$)]  $F_k=\left\langle w_1,\punkte,w_{d_k}\right\rangle$
\end{enumerate}
 and for every  $ k\in\{2,\punkte, p\}$:
\begin{enumerate} 
\item[($ \textrm{b}_k$)]  $\varphi(w_x)=\left\lbrace \begin{array}{ll} 

w_{x+1}\mod\left\langle w_{d_1+2},\punkte,w_n\right\rangle, &  \textrm{if}~x<d_1;\\
w_{x+1} \mod \left\langle w_{d_k+1},\punkte, w_n\right\rangle, &  \textrm{if}~d_{k-1}\leq x< d_k;\\
0, & \textrm{if}~ x=n.
                           \end{array}\right. $
\end{enumerate}
\end{enumerate}
\end{lemma}
\begin{proof}
\textit{If 2. holds true, then 1. follows:}\\[1ex]
Let $\{w_1,\punkte, w_n\}$ be a basis of $V$ that fulfills ($ \textrm{a}_k$) and ($ \textrm{b}_k$).\\
 An easy induction shows \[\varphi^i(w_x) = \left\lbrace
\begin{array}{ll} 
 w_{x+i} \mod \left\langle  w_j\mid j>x+i\, \right\rangle , &  \textrm{if}~x+i\leq n; \\ 
0, &  \textrm{if}~x+i> n.
                                                        \end{array}\right.\]
 Thus, 
\[\varphi^{n-d_k}(F_k)= \left\langle  \varphi^{n-d_k}(w_1),\punkte, \varphi^{n-d_k}(w_{d_k}) \right\rangle  = \left\langle  w_{n-d_k+1},\punkte, w_n \right\rangle  \] 
and $\dim_K\varphi^{n-d_k}(F_k) = d_k$ for all $k\in\{0,\punkte, p\}$.\\[2ex]
\textit{If 1. holds true, then 2. follows:}\\[1ex]
 By \cite[Theorem 5.1]{BoRe}, we find a basis $\{u_1,\punkte, u_n\}$ of $V$ that is adapted to $F_*$ and which fulfills
\[\varphi\left(u_x\right)=u_{x+1} \mod\left\langle u_{x+2},\punkte,u_n\right\rangle.\]
It is clear by the theorem of the Jordan normal form that we can modify this basis, such that
\[\varphi(u_{x}) =\left\lbrace 
\begin{array}{ll} 
u_{x+1} \mod \left\langle u_{d_k+1},\punkte, u_n\right\rangle, &  \textrm{if}~ d_{k-1}<x< d_k; \\ 
u_{d_k+1} \mod \left\langle u_{d_k+2},\punkte, u_n\right\rangle, &  \textrm{if}~x=d_k.
                                \end{array}\right.\]

Let $k\in\{2,\punkte, p\}$. 
Then there are elements $\eta_{i}\in K$, such that \[\varphi\left(u_{d_{k-1}}\right)=u_{d_{k-1}+1}+\sum\limits_{i=d_{k-1}+2}^{d_{k}}\eta_{i}\cdot u_i \mod \left\langle u_{d_{k}+1},\punkte,u_n\right\rangle\]
We define \[v'_x\coloneqq \left\lbrace
\begin{array}{ll}
u_x+ \sum\limits_{i=x+1}^{d_k}\eta_{d_{k-1}-x+1+i}\cdot v_i, &  \textrm{if}~d_{k-1}<x< d_k;\\
u_x, &  \textrm{otherwise}.\\
                           \end{array}\right.\]
Then clearly $\{v'_1,\punkte, v'_n\}$ build a basis of $V$ that is adapted to $F_*$ and
 \[\varphi\left(v'_x\right)=\left\lbrace \begin{array}{ll} 
v'_{x+1} \mod \left\langle v'_{d_k+1},\punkte, v'_n\right\rangle, &  \textrm{if}~d_{k-1}\leq x< d_k;\\
0, & \textrm{if}~ x=n. 
                           \end{array}\right.\]

We fix elements $\lambda_x\in K$, such that for $1\leq x<d_1$: 
\[\varphi\left(v'_x\right)=v'_{x+1}+\lambda_{x}\cdot v'_{d_1+1}\mod\left\langle v'_{d_1+2},\punkte,v'_n\right\rangle.\]
Let us define
 \[\mu_{x-1}\coloneqq\left\lbrace \begin{array}{ll}
1, &  \textrm{if}~x=1;\\
-\lambda_{d_1-1}, &  \textrm{if}~x=2;\\
-\sum\limits_{i=0}^{x-1}\mu_i\cdot\lambda_{d_1-x+i}, &  \textrm{if}~2<x< d_1, 
                           \end{array}\right. \]
and set
\[w_x\coloneqq\left\lbrace \begin{array}{ll}
\sum\limits_{i=0}^{d_1-x}\mu_i\cdot v'_{x+i}, &  \textrm{if}~x<d_1;\\
v'_x, &  \textrm{if}~x\geq d_1.  
                           \end{array}\right.\]
Then $\{w_1,\punkte, w_n\}$ is a basis of $V$ that is obviously adapted to $F_*$ since $v'_1,\punkte,v'_n$ is adapted to $F_*$. The claim follows.
\end{proof}

We make use of Lemma \ref{pargenlem} in order to find a generic normal form in $\N$. Therefore, given $a,b\in\{0,\punkte,n\}$ and a matrix $N\in \N$, we define $N_{(a,b)}$ to be the submatrix formed by the last $a$ rows and the first $b$ columns of $N$. 
\begin{corollary}\label{hnf} The following conditions on a matrix $N\in\N$ are equivalent:
\begin{enumerate}
\item The first $d_k$ columns of $N^{n-d_k}$ are linearly independent for $k\in\{1,\punkte,p-1\}$ ,
\item the minor $\det ((N^{n-d_k})_{(d_k,d_k)})$ is non-zero for each $k\in\{1,\punkte,p-1\}$ ,
\item $N$ is $P$-conjugate to a unique matrix $H$, such that for all $k\in\{1,\punkte,p\}$:
\[H_{i,j}=\left\lbrace \begin{array}{ll}
0, & \textrm{if}~ i\leq j; \\ 
0, &  \textrm{if}~i=d_1+1~ \textrm{and}~ j<d_1;\\
0, &  \textrm{if}~ d_{k-1}+3\leq i\leq d_k ~  \textrm{and}~ d_{k-1}+1\leq j\leq d_k-2,~ \textrm{such~that}~ i>j+1;\\
0, &  \textrm{if}~ d_{k-1}+2\leq i\leq d_{k}~ \textrm{and}~j=d_{k-1}; \\
1, &  \textrm{if}~i=j+1. \\  
                      \end{array}\right.\]
\end{enumerate}
\end{corollary}
As a direct consequence of Lemma \ref{pargenlem}, two corollaries follow.
\begin{corollary}
 The affine space
\[\Ha_B:=\{ H\in \N\mid  H_{i,j}=0~{\rm for}~i\leq j;~ H_{i+1,i}=1~{\rm for~all~}i\}\]
is a generic normal form for the $B$-action on $\N$. We denote $\N_B\coloneqq B.\Ha_B\subseteq \N$, which is an open subset.
\end{corollary}
\begin{corollary}
 The space
\[\Ha_U:=\{ H\in \N\mid  H_{i,j}=0~{\rm for}~i\leq j;~ H_{i+1,i}\neq 0~{\rm for~all~}i\}\]
is a generic normal form for the $U$-action on $\N$. We denote  $\N_U\coloneqq U.\Ha_U\subseteq \N$, which is an open subset. 
\end{corollary}
Let us end the section by giving an example.
\begin{example}
 Consider the parabolic subgroup $P\subseteq K^{9\times 9}$ given by the block sizes $(3,4,2)$.
 Then the generic $P$-normal form described in Lemma \ref{pargenlem} is given by the matrices

\[ X_0 := \left\lbrace  \left(\begin{array}{ccc|cccc|cc}
0&0&0&0&0&0&0&0& 0\\
1&0&0&0&0&0&0&0&0 \\ 
0&1&0&0&0&0&0&0&0 \\ \hline
0&0&1&0&0&0&0&0&0 \\
a_{5,1}&a_{5,2}&0&1&0&0&0&0&0 \\ 
a_{6,1}&a_{6,2}&0&0&1&0&0&0&0 \\ 
a_{7,1}&a_{7,2}&0&0&0&1&0&0&0 \\ \hline
a_{8,1}&a_{8,2}&a_{8,3}&a_{8,4}&a_{8,5}&a_{8,6}&1&0&0 \\ 
a_{9,1}&a_{9,2}&a_{9,3}&a_{9,4}&a_{9,5}&a_{9,6}&0&1&0 \\         

        \end{array}\right) ,~ {\rm where}~ a_{i,j}\in K \right\rbrace 
  \]

\end{example}

\section{Generation of (semi-) invariant rings}\label{generation}
From now on, we consider the action of the Borel subgroup $B$ and the unipotent subgroup $U$ on the nilpotent cone $\N$. We define (semi-) invariants which generate the corresponding ring of (semi-) invariants (this will be shown in Theorem \ref{genersemi}). Let us start by defining those Borel-semi-invariants introduced in \cite{BoRe}.\\[1ex]
Given $i\in\{1,\punkte,n\}$, we denote by $\omega_i\colon B\rightarrow\textbf{G}_m$ the character which is defined by $\omega_i\left(g\right)=g_{i,i}$; the $\omega_i$ form a basis for the group of characters of $B$.\\[1ex]
Let us fix integers $s,t\in\mathbf{N}$. For $i\in\{1,\punkte,s\}$ and $j\in\{1,\punkte,t\}$, we fix integers $a_i,a'_j\in\{1,\punkte,n\}$ with \mbox{$a_1+\punkte+a_s=a'_1+\punkte+a'_t=:r$} and polynomials $\Pa_{i,j}\left(x\right)\in K[x]$.\\[1ex]
Let $N\in \N$, then for all such $i$ and $j$ we consider the submatrices 
$\Pa_{i,j}\left(N\right)_{(a_i,a'_j)}\in K^{a_i\times a'_j}$  and form the $r\times r$-block matrix \[N^{\Pa}\coloneqq\left(\Pa_{i,j}\left(N\right)_{(a_i,a'_j)}\right)_{i,j},~{\rm where}~\Pa\coloneqq \left(\left(a_i\right)_i,\left(a'_j\right)_j,\left(\Pa_{i,j}\right)_{i,j}\right).\] 
The following proposition can be found in \cite[Proposition 5.3]{BoRe}.
\begin{proposition}\label{semiprop}
 For every datum $\Pa$ as above, the function 
\[f^{\Pa}\colon \N~\rightarrow~ K; ~~N~\mapsto~\det\left(N^{\Pa}\right)\]
 defines a $B$-semi-invariant regular function on $\N$ of weight \[\sum_{i=1}^s\left(\omega_{n-a_i+1}+\punkte+\omega_n\right)-\sum_{j=1}^t\left(\omega_1+\punkte+\omega_{a'_j}\right).\]
\end{proposition}
Note that the function $f^{\Pa}$ is also a $U$-invariant regular function on $\N$.
\begin{theorem}\label{genersemi}
The semi-invariant ring $K[\N]^B_*$ is generated by the semi-invariants of Proposition \ref{semiprop}.
\end{theorem}
\begin{proof}
First, we show $K[R^{\inj}_{\dfs}(\Q_n,I_x)]^{\GL_{\dfs}}_*\subseteq K[R_{\dfs}(\Q_n)]^{\GL_{\dfs}}_*$:\\[1ex]
The surjection $K[R_{\dfs}(\Q_n)]\rightarrow K[R_{\dfs}(\Q_n,I_x)]$ induces a surjection on the corresponding semi-invariant rings, since $\GL_{\dfs}$ is reductive. Furthermore, the codimension of $R_{\dfs}(\Q_n,I_x)\backslash R^{\inj}_{\dfs}(\Q_n,I_x)$ in $R_{\dfs}(\Q_n,I_x)$ is greater or equal than $2$, which yields the claim.\\[1ex] 
Following Lemma \ref{bijection}, we see that each $B$-semi-invariant $f$ on $\N$ is uniquely lifted to a $\GL_{\dfs}$-semi-invariant in $K[R^{\inj}_{\dfs}(\Q_n,I_x)]$. Theorem \ref{Schosemi} yields that $K[R_{\dfs}(\Q_n)]^{\GL_{\dfs}}_*$ is spanned by the determinantal semi-invariants $f_{\phi}$ defined in Section \ref{transsect}. Therefore, it suffices to prove that each determinantal semi-invariant, restricted to $R^{\inj}_{\dfs}(\Q_n,I_x)$, corresponds to one of the $B$-semi-invariants of Proposition \ref{semiprop}.\\[1ex]
Let us fix an arbitrary morphism in $\add \Q$, say 
\[\phi\colon \bigoplus_{j=1}^n O(j)^{x_j}\rightarrow \bigoplus_{i=1}^n O(i)^{y_i},\] such that
$h\coloneqq \sum_{j\in \Q_0}x_j \cdot j=\sum_{i\in \Q_0}y_i \cdot i.$  Then, by Section \ref{transsect}, we obtain a determinantal semi-invariant 
$f_{\phi}$.\\[1ex] 
The homomorphism spaces $P(j,i)$ between two objects $O(j)$ and $O(i)$ in $\add\Q$ are generated as $K$-vector spaces by
\[P(j,i)=\left\lbrace 
\begin{array}{ll}
0, & \textrm{if}~j>i; \\[1ex] 
\left\langle \rho_{j,i}\coloneqq\alpha_{i-1}\cdots \alpha_j\right\rangle, & \textrm{if}~j\leq i<n; \\[1ex] 
\left\langle \rho_{j,n}^{(k)}\coloneqq\alpha^k\alpha_{n-1}\cdots \alpha_j \mid k\in \mathbf{N}\cup\{0\}\right\rangle, & \textrm{if}~i=n.
                     \end{array}\right. \] 
The morphism $\phi$ is given by a $\sum_{i=1}^n y_i \times \sum_{j=1}^n x_j$-matrix $H$ with entries being morphisms between objects in $\add\Q$. 
We can view the matrix $H$ as an $n\times n$ block matrix $H=(H_{i,j})_{1\leq i,j\leq n}$ with $H_{i,j}\in K^{y_i\times x_j}$ for $i,j\in\{1,\punkte,  n\}$. Then
\[\left(H_{i,j}\right)_{k,l}= \left\lbrace
\begin{array}{ll}
0, & \textrm{if} ~ i<j; \\ 
\lambda^{k,l}_{i,j}\cdot \rho_{j,i}, & \textrm{for~some~} \lambda^{k,l}_{i,j}\in K~\textrm{if}~ j\leq i<n;\\ 
 \sum\limits_{h=0}^{\infty}\left(\lambda^{k,l}_{n,j}\right)_h\cdot \rho^{(h)}_{j,n}, &\textrm{for~some~}\left(\lambda^{k,l}_{n,j}\right)_h\in K~ \textrm{if}~ j\leq i=n.
                                        \end{array}\right.  \] 
Given an arbitrary matrix $N\in \N$, we reconsider the representation $M^N$ defined in Lemma \ref{bijection}.
 Since $\GL_{\dfs}$ acts transitively on $R_{\dfs}^{\inj}(\Q')$ with $\Q'$ being the linearly oriented quiver of Dynkin type $A_n$, we can examine the restricted semi-invariant on these representations $M^N$.\\[1ex]
The $B$-semi-invariant of $\N$ associated to $f_{\phi}$ via the translation of Lemma \ref{bijection} is given by
\[f^{\phi}\colon \N \rightarrow K;~ N \mapsto \det M^N(\phi).\]
The matrix 
\[M^N(\phi)=\left(M^N_{i,j}\right)_{1\leq i,j\leq n}\in K^{h\times h}\]
 is given as a block matrix
where each block 
\[M^N_{i,j}=  \left(\left(M^N_{i,j}\right)_{k,l}\right)_{\begin{subarray}{l}
1\leq k\leq y_i \\ 
1\leq l\leq x_j \end{subarray}}\in K^{iy_i\times jx_j}\]
is again a block matrix. The blocks of $M^N_{i,j}$ are given by
\[K^{i\times j} \ni \left(M^N_{i,j}\right)_{k,l}=\left\lbrace \begin{array}{ll}
0, & \textrm{if} ~ i<j; \\ 
\lambda^{k,l}_{i,j}\cdot E^{(i)}_{(i,j)}, & \textrm{if}~ j\leq i<n;\\ 
 \sum\limits_{h=0}^{\infty}\left(\lambda^{k,l}_{n,j}\right)_h\cdot \left(N^h\right)_{(n,j)}, & \textrm{if}~ j\leq i=n; 
                                        \end{array}\right.\]
such that $E^{(i)}\in K^{i\times i}$ is the identity matrix.
Note that if  $i,j\in\{1,\punkte,n\}$ and $i<n$, then $M_{i,j}^N=M_{i,j}^{N'}=:M_{i,j}$ for every pair of matrices $N,N'\in \N$.\\[1ex]
We can without loss of generality assume $y_1=\punkte=y_{n-1}=0$ which can, for example, be seen by induction on the index $i$ of $y_i$. This assumption is not necessary for the proof, but will shorten the remaining argumentation. Let us
define
 \[a\coloneqq(\underbrace{n,\punkte,n}_{=:a_{1},\punkte,a_{y_n}})~~~
\textrm{and}~~~a'\coloneqq(\underbrace{1,\punkte,1}_{=:a'_{1,1},\punkte,a'_{1,x_1}},\underbrace{2,\punkte,2}_{=:a'_{2,1},\punkte,a'_{2,x_2}},\punkte,\underbrace{n,\punkte,n}_{=:a'_{n,1},\punkte,a'_{n,x_n}}).\]
Furthermore, define for $j\in\{1,\punkte, n\}$ and for each pair of integers $k\in \{1,\punkte, y_n\}$ and $l\in\{1,\punkte, x_j\}$ the polynomial
 \[\Pa^{(k,l)}_{j}\coloneqq\sum_{h=0}^{\infty}\left(\lambda^{k,l}_{n,j}\right)_h\cdot X^{h}. \]
Let us denote $\Pa\coloneqq\left(a,a',\left(P_{j}^{(k,l)}\right)_{j,k,l}\right)$ and let $N\in \N$; it suffices to show $f^{\Pa}(N)=f^{\phi}(N)$:
\begin{align}
     f^{\phi}(N)~  ~=~\det M^N(\phi)&  ~=~ \det \left(M^N_{n,j}\right)_{1\leq j\leq n} ~=~ \det \left( \left(\left(M^N_{n,j}\right)_{k,l}\right)_{\begin{subarray}{l}
1\leq k\leq y_n \\ 
1\leq l\leq x_j \end{subarray}}\right)_{1\leq j\leq n}\nonumber\\
           & ~=~ \det \left(\left(\sum\limits_{h=0}^{\infty}\left(\lambda^{k,l}_{n,j}\right)_h\cdot \left(N^h\right)_{(n,j)}\right)_{\begin{subarray}{l}
1\leq k\leq y_n \\ 
1\leq l\leq x_j \end{subarray}}\right)_{1\leq j\leq n} \nonumber\\
 &~=~ \det \left(\left(\Pa^{\left(k,l\right)}_{j}(N)_{(n,j)}\right)_{\begin{subarray}{l}
1\leq k\leq y_n \\ 
1\leq l\leq x_j \end{subarray}}\right)_{1\leq j\leq n}~=~ \det N^{\Pa}~=~ f^{\Pa}(N) .\qedhere \nonumber
\end{align}
\end{proof}
\begin{corollary}\label{Ugeninv}
 The $U$-invariant ring $K[\N]^U$ is spanned by the induced $U$-invariants.
\end{corollary}
\section[About the algebraic U-quotient of the nilpotent cone]{About the algebraic $U$-quotient of the nilpotent cone}\label{Uquot}
We have seen that the $U$-invariant ring $K[\N]^U$ is spanned by the functions defined in Proposition \ref{semiprop}. We prove a quotient criterion in the next subsection which will help to provide the explicit structure of the $U$-invariant rings for the cases $n=2,3$. 
\subsection{A quotient criterion}
Let $G$ be a reductive algebraic group and $U$ be a unipotent subgroup. Then $U$ acts on $G$ by right multiplication and Lemma \ref{Uinvfin} states that the $U$-invariant ring $K[G]^U$ is finitely generated as a $K$-algebra. Thus, an algebraic $U$-quotient of $G$, namely $G\quot U\coloneqq\Spec K[G]^U$, exists together with a dominant morphism $\pi_{G\quot U}\colon G\rightarrow G\quot U$ which is in general not surjective. Note that there is an element $\overline{e}\in G\quot U$, such that  $\pi_{G\quot U}(g)=g\overline{e}$ for all $g\in G$.\\[1ex]
The group $G$ acts on $G\quot U$ by left multiplication. Let $X$ be an affine $G$-variety and consider the diagonal operation of $G$ on the affine variety $G\quot U\times X$; we consider the natural $G$-equivariant morphism $\iota: X\rightarrow G\quot U\times X$.\\
 Let $\pi'\colon G\quot U\times X\rightarrow (G\quot U\times X)\quot G\coloneqq\Spec K[G\quot U\times X]^G$ be the associated algebraic $G$-quotient, then we obtain a morphism 
\[\rho\coloneqq \pi'\circ \iota: X\rightarrow  (G\quot U\times X)\quot G.\]
The morphism $\rho$ induces an isomorphism 
$\rho^*\colon (K[G]^U\otimes K[X])^G \rightarrow K[X]^U.$\\
Thus, $X\quot U\cong (G\quot U\times X)\quot G$ and
\[K[X]^U \cong (K[G\quot U\times X])^G\cong (K[G\quot U]\otimes K[X])^G\cong (K[G]^U\otimes K[X])^G.\]
Let $Y$ be an affine $G$-variety and let $\mu': G\quot U\times X\rightarrow Y$ be a $G$-invariant morphism, together with a dominant $U$-invariant morphism of affine varieties 
\[\mu\colon X \rightarrow Y;~x  \mapsto (f_{1}(x),\punkte, f_s(x)),\] such that $\mu'\circ \iota = \mu$.\\[1ex]
In this setting, we obtain the following criterion for $\mu$ to be an algebraic $U$-quotient.
\begin{lemma}\label{critU}Assume that
\begin{enumerate}
 \item[(1.)] $Y$ is normal,
 \item[(2.)] $\mu$ separates the $U$-orbits generically, that is, there is an open subset $Y_U\subseteq Y$, such that $\mu(x)\neq \mu(x')$ for all $x,x'\in X_U\coloneqq\mu^{-1}(Y_U)$, and
\item[(3.)] $\codim_Y(\overline{Y\backslash Y_U})\geq 2$ or $\mu$ is surjective.
\end{enumerate}
Then $\mu$ is an algebraic $U$-quotient of $X$, that is, $Y\cong X\quot U$.
\end{lemma}
\begin{proof}
Let $g_{1},\punkte,g_s\in K[G\quot U\times X]^G$, such that $\rho^*(g_{i})=f_{i}$ for all $i$.\\[1ex]
Clearly,
\[2\leq \codim_Y(\overline{Y\backslash \mu(X)})\leq\codim_{Y}(\overline{Y\backslash \mu'(G\quot U\times X)}).\]
The morphism $\mu'$ separates the $G$-orbits in  $ G\quot U\times X$ generically (that is, in $Y_U$):\\
If $x,x'\in G.(\{\overline{e}\}\times X_U)$, then $\mu'(x)\neq \mu'(x')$. The morphism $\mu'$ restricts to a surjection $G.(\{\overline{e}\}\times X_U)\rightarrow Y_U$, furthermore, the algebraic quotient $\pi'$ is surjective and there exists a morphism $\tilde{\mu'}: (G\quot U\times X)\quot G\rightarrow Y$, such that $\tilde{\mu'}\circ \pi'=\mu'$. Then $Y_U\subseteq \im(\tilde{\mu'})$ and, since each fibre of $\pi'$  contains exactly one closed $G$-orbit, we have shown that generically each fibre of $\mu'$ contains a unique   closed orbit. \\[1ex]
Thus, Theorem \ref{criterion} yields that $\pi\colon G\quot U\times X\rightarrow Y$ is an algebraic $G$-quotient. Since $f_{i}$ and $g_{i}$ correspond to each other via the isomorphism $ \rho^*\colon (K[G]^U\otimes K[X])^G \rightarrow K[X]^U,$ the morphism $\mu\colon X\rightarrow Y$ is an algebraic $U$-quotient of $X$.
\end{proof}
We are now able to give explicit descriptions of algebraic $U$-quotients of the nilpotent cone in case $n$ equals $2$ or $3$.
\begin{example}\label{Utwo}
 We consider the case $n=2$. In this case, the $U$-normal form of Section \ref{gnfsect} is given by matrices  
 \[H_{x}\coloneqq\left( \begin{array}{ll}
0 &0  \\ 
x & 0
\end{array}\right)\]
 where $x\in K^*$. Then by Proposition \ref{semiprop}, we define the $U$-invariant $f_{2,1}$ by $f_{2,1}(N)=N_{2,1}$ for $N=(N_{i,j})_{i,j}\in\N$.\\[1ex]
 The morphism 
\[\mu\colon \N\rightarrow \textbf{A}^1= \Spec K[f_{2,1}];~N~\mapsto f_{2,1}(N)\]
is an algebraic $U$-quotient of $\N$:\\[1ex]
Clearly, the variety $\textbf{A}^1$ is normal and $\mu$ separates the $U$-orbits in the open subset $\N_U\subseteq \N$.
 Since $\mu$ is surjective, Lemma \ref{critU} yields the claim. We have, therefore, proven 
\[K[\N]^U=K[f_{2,1}].\]
\end{example}
The case $n=3$ is slightly more complex, but can still be handled by making use of Lemma \ref{critU}.
\begin{example}\label{Uthree}
 In  case $n=3$,  the $U$-normal forms are given by matrices \[H=\left( \begin{array}{lll}
0&0 &0  \\ 
x_1&0 & 0\\
x& x_2& 0
\end{array}\right),~x_1,x_2\in K^*.\] Following Proposition \ref{semiprop}, we define certain $U$-invariants; consider $N=(N_{i,j})_{i,j}\in \N$, then $f_{3,1}(N)= N_{3,1}$, ${\det}_1(N)= N_{2,1}N_{3,2}-N_{2,2}N_{3,1}$ and ${\det}_2(N) =N_{1,1}N_{3,1}+N_{2,1}N_{3,2}+N_{3,1}N_{3,3}$. Note that the equality $\Det_1(N)=\Det_2(N)$ holds true for all $N\in\N$ due to the nilpotency conditions.\\[1ex]
Furthermore, we define a $U$-invariant $f_1$ given by the datum $\Pa=((2),(1,1), (x,x^2))$, thus, 
 $f_{1}(N)= N_{2,1}\cdot {\det}_1+N_{3,1}\cdot(N_{2,1}N_{3,3}-N_{3,1}N_{2,3})$.\\[1ex]
And the $U$-invariant $f_2$ given by the datum $\Pa=((1,1),(2),(x^2,x)$, thus,
 $f_{2}(N) = N_{3,2}\cdot {\det}_1+N_{3,1}\cdot(N_{1,1}N_{3,2}-N_{1,2}N_{3,1})$. Then $f_1\cdot f_2=\Det_1^3$ holds true in $\N$.\\[1ex]
\textbf{Claim:}
 The morphism 
 \begin{align}
 \mu\colon& ~\N\rightarrow \textbf{A}^1\times \Spec \dfrac{K[X_1,X_2,Z]}{\left( X_1X_2=Z^3\right)}=: Y\nonumber\\
&~N~\mapsto (f_{3,1}(N), f_1(N),f_2(N), {\det}_1(N)) \nonumber
\end{align}
is an algebraic $U$-quotient of $\N$.\\[2ex]
The affine variety $Y$ is normal as the product of $\textbf{A}^1$ and a normal toric affine variety induced by the strongly convex rational polyhedral cone \[\sigma\coloneqq \Cone\left(
\left(\begin{array}{l}
1 \\ 
1     \end{array}\right),\left(\begin{array}{l}
1 \\ 
2     \end{array}\right),\left(\begin{array}{l}
2 \\ 
1     \end{array}\right)
 \right). \] 
The morphism $\mu$ separates the $U$-orbits in the open subset $\N_U\subseteq \N$ as can be proved by a direct calculation.\\[1ex]
Furthermore, $\codim_{Y}(\overline{Y\backslash \mu(\N)})\geq 2$, since $\textbf{A}^1\times  X'\subset \mu(\N)$ and $(s,t,u,v)\in \mu(\N)$ whenever either $s$, $t$ or $u$ equals zero and $v^3=ut$.
Lemma \ref{critU} yields the claim.\\[1ex]
We have proved 
\[K[\N]^U= K[f_{3,1}, f_1, f_2, {\det}_1]/\left( f_1\cdot f_2= {\det}_1^3\right).\]
\end{example}

\subsection{Toric invariants}\label{toricsect}
As the case $n=3$ suggests, there is a toric variety closely related to $\N\quot U$.\\[1ex]
The idea of a generalization is the following: By considering a special type of $U$-invariants, so-called toric invariants, we define a toric variety $X$ together with a dominant morphism
$ \N\quot U\rightarrow X$, such that the generic fibres are affine spaces of the same dimension.\\[1ex]
Given a matrix $H=(x_{i,j})_{i,j}\in \Ha_U$, we denote $x_i\coloneqq x_{i+1,i}$ and define its \textit{toric part} $H_{\tor}\in K^{n\times n}$ by 
\[(H_{\tor})_{i,j}\coloneqq\left\lbrace
\begin{array}{ll}
x_{i}, & \textrm{if}~i=j+1; \\ 
0, & \textrm{otherwise}.
                                   \end{array}\right. \] 
Let us fix an invariant $f$ of size $r$. It is called \textit{toric} if $f(H)=f(H_{\tor})$ for every matrix $H\in \Ha_U$ and \textit{sum-free} if its block sizes  $a_{1},\punkte, a_{s}$ and $a'_{1}, \punkte , a'_{t}$ do not share any partial sums, that is, $\sum_{i\in I} a_i \neq \sum_{i'\in I'}a'_{i'}$ for all $I\subsetneq \{1,\punkte,s\}$ and $I'\subsetneq \{1,\punkte,t\}$.\\[1ex]
For $k\in\{ 1,\punkte, s\}$ we denote the \textit{horizontal change} of $k$ by $\hc(k)$, that is, the minimal integer, such that there is an integer $\hs(k)>0$ (the \textit{horizontal split}) with 
 \[\sum_{j=1}^k a_{j}  = \sum_{j=1}^{\hc(k)} a'_{j}-\hs(k).\]
We denote the \textit{complement of $\hs(k)$} by $\ch(k)\coloneqq a'_{\hc(k)}-\hs(k)$; for formal reasons, we define $\hc(0)\coloneqq 0$.\\[1ex] 
For $k\in\{1,\punkte, t\}$ denote  the \textit{vertical change} by $\vc(k)$, that is, the minimal integer, such that there is an integer $\vs(k)>0$ (the \textit{vertical split}) with 
 \[\sum_{j=1}^k a'_{j}  = \sum_{j=1}^{\vc(k)} a_{j}-\vs(k).\]
We denote the  \textit{complement of $\vs$} by $\cv(k)\coloneqq a'_{\vc(k)}-\vs(k)$; for formal reasons we define $\vc(0)\coloneqq 0$ as above.\\[1ex]
 For every $i\in\{1,\punkte, r\}$, we define
the \textit{horizontal block} $\hb(i)$, that is, the maximal integer with $i=\sum\limits_{j=1}^{\hb(i)-1} a_{j} +\hd(i)$ for a positive integer $\hd(i)$ (the \textit{horizontal datum}). Analogously, we define
the \textit{vertical block} $\vb(i)$, that is, the maximal integer with $i=\sum\limits_{j=1}^{\vb(i)-1} a'_{j} +\vd(i)$ for a positive integer $\vd(i)$ (the \textit{vertical datum}).\\[1ex]
Let us call an entry $(i,j)\in \{1,\punkte,r\}^2$ \textit{acceptable} for $(\underline{a},\underline{a}')$ if $\vd(j)< \hd(i)+n -a_{\hb(i)}$ and \textit{unacceptable} otherwise.\\[1ex]
 A permutation $\sigma\in S_r$ is called \textit{acceptable} for $(\underline{a},\underline{a}')$ if every entry $(i,\sigma(i))$ is acceptable for $f$.\\[1ex]
There exists a minimal, finite set $\{f_1,\punkte, f_s\}$ of toric invariants that generates all toric invariants, such that for each $i\in\{1,\punkte,s\}$, there are integers $h_1,\punkte,h_{n-1}$ with \[f(H)=x_1^{h_1}\cdot \punkte \cdot x_{n-1}^{h_{n-1}}. \]
The set $S$ of these tuples $(h_1,\punkte, h_{n-1})$ yields a cone $\sigma=\Cone(S)$, such that the variety $X\coloneqq \Spec KS_{\sigma}$ is the aforementioned toric variety.
The proof of the following lemma and the notion of an acceptable permutation yield that we can calculate a toric invariant if we have one acceptable permutation for its block sizes.

\begin{lemma}\label{secred}
The toric invariants are generated by the sum-free toric invariants.
\end{lemma}

\begin{proof}
First, we reduce the problem as follows:\\
\textbf{Claim 1:} Let $\sigma\in S_r$ be a permutation, such that $\prod\limits_{i=1}^r (H^{\Pa})_{i,\sigma(i)}\neq 0$ for every $H\in \Ha_U$. Then there is an element $\lambda\in K^*$, such that for every $H\in \Ha_U$
 \[f(H)=\lambda\cdot \prod_{i=1}^r (H^{\Pa})_{i,\sigma(i)}.\]
\begin{proof}[Proof of Claim 1]
 Since
every permutation equals a product of transpositions, thus, it suffices to show that for every choice $1\leq i,i',j,j'\leq r$ with $H^{\Pa}_{i,j}\cdot H^{\Pa}_{i',j'}\neq 0 \neq  H^{\Pa}_{i,j'}\cdot H^{\Pa}_{i',j},$ there is an element $\lambda\in K^*$, such that
\[H^{\Pa}_{i,j}\cdot H^{\Pa}_{i',j'}= \lambda \cdot H^{\Pa}_{i,j'}\cdot H^{\Pa}_{i',j}.\]

Since $H^{\Pa}=((P_{i,j}(H))_{(a_i,a'_j)})_{\begin{subarray}{l}
1\leq i\leq s \\ 
1\leq j\leq t \end{subarray}}$, there are 
integers $s',s''\in \{1,\punkte, s\}$ and $t',t''\in\{1,\punkte, t\}$ and integers
 $x^{'}\in\{1,\punkte, a_{s^{'}}\}$ and $x^{''}\in\{1,\punkte, a_{s^{''}}\}$, as well as integers $ y^{'}\in\{1,\punkte,  a'_{t^{'}}\}$ and $ y^{''}\in\{1,\punkte,  a'_{t^{''}}\}$,
such that 
\[H^{\Pa}_{i,j}=(P_{s',t'}(H))_{(a_{s'},a'_{t'})})_{x',y'},~ H^{\Pa}_{i',j'}=(P_{s'',t''}(H))_{(a_{s''},a'_{t''})})_{x'',y''},\]
\[H^{\Pa}_{i,j'}=(P_{s',t''}(H))_{(a_{s'},a'_{t''})})_{x',y''} ~\textrm{and} ~H^{\Pa}_{i',j}=(P_{s'',t'}(H))_{(a_{s''},a'_{t'})})_{x'',y'}.\]

Following the above considerations, there are elements $\mu_1,\mu_2,\mu_3,\mu_4\in K^*$, such that 
\[ H^{\Pa}_{i,j}\cdot H^{\Pa}_{i',j'}      = \frac{\mu_1\cdot\mu_2}{\mu_3\cdot \mu_4}\cdot  H^{\Pa}_{i,j'} \cdot H^{\Pa}_{i',j}\]
which yields the claim.
\end{proof}
In order to calculate a set of minimal generators, we can without loss of generality assume $a_i,a'_j\leq n-1$ for all $i\in\{1,\punkte, s\}$ and $j\in\{1,\punkte, t\}$, since otherwise the corresponding semi-invariant $f$ fulfills $f(H)=0$ for every $H\in \Ha_U$ or deletion of these blocks leads to changing $f$ by a scalar.\\[1ex]
Let $f$ be a toric $U$-invariant. To see of which form $f$ is on $\Ha_U$, we can without loss of generality order $\underline{a}:=(a_{1},\punkte, a_{s})$ and $\underline{a}:=(a'_{1}, \punkte , a'_{t})$ as we like and adapt the permutation accordingly.\\[1ex]
 It, therefore, suffices to consider an arbitrary $r\times r$-matrix of sum-free block sizes $\underline{a}$ and $\underline{a}'$.
If we find an acceptable permutation $\sigma$ for $(\underline{a},\underline{a}')$, following the above claim there exists an element  $\mu\in K^*$ and a datum $\Pa$ which fulfills $f(H)=\mu\cdot f^{\Pa}(H)=\mu\cdot \prod\limits_{i=1}^r (H^{\Pa})_{i,\sigma(i)}\neq 0$  for every $H\in \Ha_U$. Then $f$ and $f^{\Pa}$ coincide generically.\\[1ex]
We define a permutation $\sigma\in S_r$, such that every $(i,\sigma(i))$ is acceptable for $(\underline{a},\underline{a}')$ by double induction on $s$ and $t$.\\[1ex]
Let $s=1$ and $t=1$, then every entry $(i,i)$ is acceptable  for $(\underline{a},\underline{a}')$, since $a_{\hb(i)}=a_1\leq n-1$ and, therefore,
\[ \vd(i)= i <  i+n-a_1  = \hd(i) +n -a_1 .\]
Let $t=1$ and assume that for every $k\leq s$, the above claim holds true. Consider the block sizes $\underline{a}:=(a_1,\punkte, a_{s+1})$ and $\underline{a}':=a'_1$, then every entry $(i,i)$ is acceptable  for $(\underline{a},\underline{a}')$, since
\[ \vd(i) = i < i+n-a_{\hb(i)} \leq \hd(i) +n -a_{\hb(i)}.\]
Let $s=1$ and assume for every $k\leq t$, the above claim holds true. Consider the block sizes $\underline{a}:=(a_1)$ and $\underline{a}':=(a'_1,\punkte, a'_{t+1})$, then every $(i,i)$ is acceptable  for $(\underline{a},\underline{a}')$ in the same way:
\[ \vd(i) = i <  i+n-a_1 = \hd(i) +n -a_1 .\]
We can set $\sigma=id$ in every of these cases.\\[1ex]
Let us fix an arbitrary integer $t$ and let us assume that for $s'\leq s$ and for every choice of block sizes $a_1,\punkte,a_{s'}$ and $a'_1,\punkte,a'_{t}$ with $\sum\limits_{j=1}^{s'}a_j=\sum\limits_{j=1}^{t} a'_j$, there is a permutation $\sigma$ as claimed.\\
We consider block sizes $\underline{a}:=(a_1,\punkte, a_{s+1})$ and $\underline{a}':=(a'_1,\punkte,a'_{t})$ with $\sum\limits_{j=1}^{s+1}a_j=\sum\limits_{j=1}^t a'_j=r$ and show in the following that we can find a permutation as wished for.

\textbf{First case:} We can order the block sizes $a'_1,\punkte, a'_t$, such that $a'_t\geq a_{s+1}$.\\[1ex]
We can apply the premise of the induction to the $r-a_{s+1}\times r-a_{s+1}$-upper-left submatrix of block sizes $\underline{a(s)}:= (a_1,\punkte, a_s)$ and $\underline{a(s)}':=(a'_1,\punkte, a'_{t-1},a'_t-a_{s+1})$ and obtain a permutation $\sigma'\in S_{r-a_{s+1}}$, such that $(i,\sigma'(i))$ is acceptable for $(\underline{a(s)},\underline{a(s)}')$ for every $i\leq r-a_{s+1}$.\\[1ex] 
We define $\sigma\in S_r$ by
\[\sigma(i)\coloneqq\left\lbrace 
\begin{array}{ll}
\sigma'(i), &\textrm{if}~i\leq r-a_{s+1};  \\ 
i, & \textrm{otherwise}. 
                         \end{array}\right. \]
Then every entry $(i,\sigma(i))$, where $i\leq r-a_{s+1}$, is acceptable  for $(\underline{a},\underline{a}')$, since it is acceptable for $(\underline{a(s)},\underline{a(s)}')$.\\[1ex]
Every entry $(i,i)$, where $i>r-a_{s+1}$, is acceptable  for $(\underline{a},\underline{a}')$, since
\[\vd(i)~=~i-\sum_{j=1}^{t-1}a'_j~<~i-\sum_{j=1}^t a'_j+ n ~=~i-\sum_{j=1}^{s}a_j+n-a_{s+1}~=~\hd(i) +n -a_{s+1}.\]
\textbf{Second case:} The inequality $a'_i< a_{j}$ holds true for every $i\in\{1,\punkte, s+1\}$ and $j\in\{1,\punkte, t\}$.\\[1ex]
\underline{\underline{Claim:}} For every $k\in\{1,\punkte, s\}$, there is a permutation $\sigma\in S_{a_1+\punkte+a_{k+1}}$, such that every entry\\
 $~~~~~~~~~~~~~(i,\sigma(i))$ is acceptable for $(\underline{a},\underline{a}')$. Furthermore, the entry $(i,i)$ is acceptable for\\
 $~~~~~~~~~~~~~(\underline{a},\underline{a}')$ for every integer $a_1+\punkte +a_{k}+\hs(k)<i\leq a_1+\punkte +a_{k+1}$.\\[1ex]
We prove the claim by induction on $k$.\\[1ex]
\underline{Let $k=1$.}\\
Define
\[\sigma(i)\coloneqq\left\lbrace 
\begin{array}{ll}
i, &\textrm{if}~i\leq a_1-\ch(1);  \\ 
i+\hs(1), &\textrm{if}~a_1-\ch(1) <	i \leq a_1;   \\ 
i-\ch(1), &\textrm{if}~a_1 < i \leq a_1+\hs(1);  \\ 
i, & \textrm{otherwise}.
                         \end{array}\right.\]
The permutation $\sigma$ can be vizualized as follows:
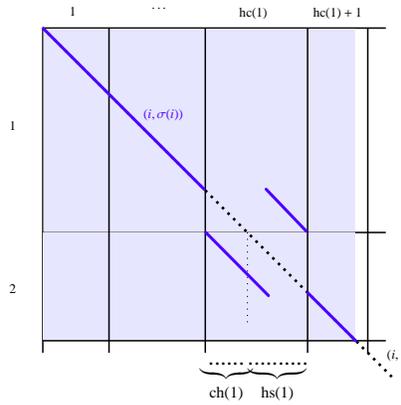
\begin{center}
\begin{tikzpicture}[scale=0.8]
   \tikzstyle{every node} =[shape=rectangle];
      \draw[line width=0.02cm] (0cm,2cm) -- (5.7cm,2cm);
      \node[anchor = north] at (-0.5cm,0.6cm)  {\tiny{$1$}};
     \node[anchor = north] at (-0.5cm,-2.1cm)  {\tiny{$2$}};
      \draw[line width=0.02cm] (0cm,-1.4cm) -- (5.7cm,-1.4cm);
 \draw[line width=0.02cm] (0cm,-3.2cm) -- (5.7cm,-3.2cm);
      \draw[line width=0.02cm] (0cm, 2cm) -- (0cm, -3.4cm);
\fill[color = lllllllightblue] (2.71cm,1.99cm) -- (4.39cm,1.99cm) -- (4.39cm,-1.39cm) -- (2.71cm,-1.39cm) -- (2.71cm,2cm);
\fill[color = lllllllightblue] (2.69cm,1.99cm) -- (1.11cm,1.99cm) -- (1.11cm,-1.39cm) -- (2.69cm,-1.39cm) -- (2.69cm,2cm);
\fill[color = lllllllightblue] (0.01cm,1.99cm) -- (1.09cm,1.99cm) -- (1.09cm,-1.39cm) -- (0.01cm,-1.39cm) -- (0.01cm,2cm);
\fill[color = lllllllightblue] (5.19cm,1.99cm) -- (4.41cm,1.99cm) -- (4.41cm,-1.39cm) -- (5.19cm,-1.39cm) -- (5.19cm,2cm);
\fill[color = lllllllightblue] (5.19cm,-1.41cm) -- (4.41cm,-1.41cm) -- (4.41cm,-3.19cm) -- (5.19cm,-3.19cm) -- (5.19cm,-1.41cm);
\fill[color = lllllllightblue] (2.71cm,-1.41cm) -- (4.39cm,-1.41cm) -- (4.39cm,-3.19cm) -- (2.71cm,-3.19cm) -- (2.71cm,-1.41cm);
\fill[color = lllllllightblue] (2.69cm,-1.41cm) -- (1.11cm,-1.41cm) -- (1.11cm,-3.19cm) -- (2.69cm,-3.19cm) -- (2.69cm,-1.41cm);
\fill[color = lllllllightblue] (0.01cm,-1.41cm) -- (1.09cm,-1.41cm) -- (1.09cm,-3.19cm) -- (0.01cm,-3.19cm) -- (0.01cm,-1.41cm);

      \draw[line width=0.02cm] (1.1cm, -3.4cm) -- (1.1cm, 2cm);
      \draw[line width=0.02cm] (2.7cm, -3.4cm) -- (2.7cm, 2cm);
      \draw[line width=0.02cm] (4.4cm, -3.4cm) -- (4.4cm, 2cm);
   \draw[line width=0.02cm] (5.4cm, -3.4cm) -- (5.4cm, 2cm);
     \node[anchor = north] at (0.5cm,2.5cm)  {\tiny{$1$}};
\node[anchor = north] at (1.95cm,2.5cm)  {\tiny{$\cdots$}};
    \node[anchor = north] at (3.5cm,2.5cm)  {\tiny{$\hc(1)$}};
   \node[anchor = north] at (3.05cm,-3.4cm)  {\small{$\underbrace{......}_{\ch(1)}$}};
 \node[anchor = north] at (3.9cm,-3.4cm)  {\small{$\underbrace{..........}_{\hs(1)}$}};
      \node[anchor = north] at (4.9cm,2.5cm)  {\tiny{$\hc(1)+1$}};
     \node[anchor = north,lblue] at (2cm,0.8cm)  {\tiny{$(i,\sigma(i))$}};
\node[anchor = north] at (5.9cm,-3.2cm)  {\tiny{$(i,i)$}};
      \draw[dotted,line width=0.035cm] (0cm, 2cm) -- (5.8cm, -3.8cm);
\draw[dotted] (3.4cm, -1.4cm) -- (3.4cm, -2.9cm);
 \draw[lblue,line width=0.04cm,line cap=round] (0.005cm, 1.995cm) -- (2.68cm, -0.68cm);
  \draw[lblue,line width=0.04cm,line cap=round] (3.71cm, -0.68cm) -- (4.38cm, -1.38cm);
  \draw[lblue,line width=0.04cm,line cap=round] (2.71cm, -1.4cm) -- (3.75cm, -2.45cm); 
\draw[lblue,line width=0.04cm,line cap=round] (5.2cm, -3.2cm) -- (4.4cm, -2.4cm); 
\end{tikzpicture}
\end{center}
For $i\leq a_1-\ch(1) $, the entry $(i,i)$ is acceptable  for $(\underline{a},\underline{a}')$ due to the considerations in the case $s=1$.\\[1ex]
For $a_1-\ch(1) <	i \leq a_1$, the entry $(i,\sigma(i))$ is acceptable  for $(\underline{a},\underline{a}')$, since
\[ \vd(\sigma(i)) = i +\hs(1) -\sum\limits_{j=1}^{\hc(1)-1}a'_j  <           i+n-a_{1}  =   \hd(i) +n -a_{1}.\]
For $a_1 < i \leq a_1+\hs(1)$, the entry $(i,\sigma(i))$ is acceptable  for $(\underline{a},\underline{a}')$, since
\[  \vd(\sigma(i)) =  i+\hs(1)-\sum\limits_{j=1}^{\hc(1)}a'_j =   i -a_1  <     i-a_1+n-a_{2} = \hd(i) +n -a_{2}.\]
For $ i >  a_1+\hs(1)$, the entry $(i,\sigma(i))$ is acceptable  for $(\underline{a},\underline{a}')$, since
\[ \vd(\sigma(i)) =  i-\sum\limits_{j=1}^{\vb(\sigma(i))-1}a'_j  < i-a_1+n-a_{2}  =  \hd(i) +n -a_{2} .\]
\underline{Now let $k+1>1$.}\\[1ex]
Assume the claim holds true for $k$, that is, there is a permutation $\sigma'\in S_{a_1+\punkte+a_{k+1}}$, such that every entry $(i,\sigma'(i))$ is acceptable  for $(\underline{a},\underline{a}')$ and such that $\sigma'(i)=i$ for every integer $a_1+\punkte +a_{k}+\hs(k)<i\leq a_1+\punkte +a_{k+1}$.\\[1ex]
 Then we set
\[\sigma(i)\coloneqq\left\lbrace 
\begin{array}{ll}
\sigma'(i), &\textrm{if}~i\leq \sum_{j=1}^{k+1} a_j-\ch(k+1) ;  \\ 
i+\hs(k+1), &\textrm{if}~\sum_{j=1}^{k+1} a_j-\ch(k+1) <	i \leq\sum\limits_{j=1}^{k+1} a_j;   \\ 
i-\ch(k+1), &\textrm{if}~\sum\limits_{j=1}^{k+1} a_j < i \leq \sum\limits_{j=1}^{k+1} a_j+\hs(k+1);  \\ 
i, & \textrm{otherwise}. 
                         \end{array}\right.\]
 As in the case $k=1$, the permutation $\sigma$ can be vizualized by\\
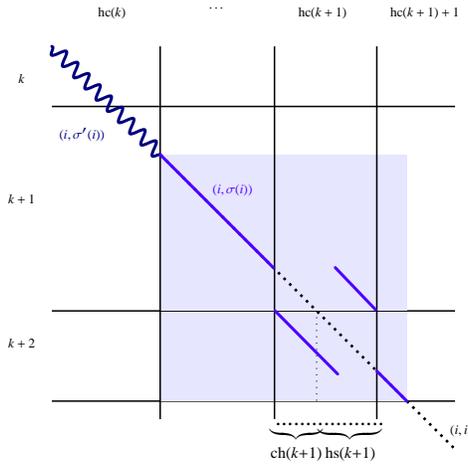
\begin{center}
\begin{tikzpicture}[scale=0.8]
   \tikzstyle{every node} =[shape=rectangle];
      \draw[line width=0.02cm] (-1cm,2cm) -- (5.7cm,2cm);
      \draw[line width=0.02cm] (-1cm,-1.4cm) -- (5.7cm,-1.4cm);
      \draw[line width=0.02cm] (-1cm,-2.9cm) -- (5.7cm,-2.9cm);
\fill[color = lllllllightblue] (2.69cm,1.195cm) -- (0.81cm,1.195cm) -- (0.81cm,-1.39cm) -- (2.69cm,-1.39cm) -- (2.69cm,1.195cm);
\fill[color = lllllllightblue] (2.71cm,1.195cm) -- (4.39cm,1.195cm) -- (4.39cm,-1.39cm) -- (2.71cm,-1.39cm) -- (2.71cm,1.195cm);
\fill[color = lllllllightblue] (4.9cm,1.195cm) -- (4.41cm,1.195cm) -- (4.41cm,-1.39cm) -- (4.9cm,-1.39cm) -- (4.9cm,1.195cm);
\fill[color = lllllllightblue] (2.69cm,-1.41cm) -- (0.81cm,-1.41cm) -- (0.81cm,-2.89cm) -- (2.69cm,-2.89cm) -- (2.69cm,-1.41cm);
\fill[color = lllllllightblue] (2.71cm,-1.41cm) -- (4.39cm,-1.41cm) -- (4.39cm,-2.89cm) -- (2.71cm,-2.89cm) -- (2.71cm,-1.41cm);
\fill[color = lllllllightblue] (4.9cm,-1.41cm) -- (4.41cm,-1.41cm) -- (4.41cm,-2.89cm) -- (4.9cm,-2.89cm) -- (4.9cm,-1.41cm);
      \draw[line width=0.02cm] (0.8cm, -3.2cm) -- (0.8cm, 3cm);
      \draw[line width=0.02cm] (2.7cm, -3.2cm) -- (2.7cm, 3cm);
      \draw[line width=0.02cm] (4.4cm, -3.2cm) -- (4.4cm, 3cm); 
\node[anchor = north] at (-1.5cm,2.7cm)  {\tiny{$k$}};
      \node[anchor = north] at (-1.5cm,0.7cm)  {\tiny{$k+1$}};
     \node[anchor = north] at (-1.5cm,-1.7cm)  {\tiny{$k+2$}};
     \node[anchor = north] at (0cm,3.8cm)  {\tiny{$\hc(k)$}};
\node[anchor = north] at (1.75cm,3.8cm)  {\tiny{$\cdots$}};
    \node[anchor = north] at (3.5cm,3.8cm)  {\tiny{$\hc(k+1)$}};
   \node[anchor = north] at (3.06cm,-3.1cm)  {\small{$\underbrace{.......}_{\ch(k+1)}$}};
 \node[anchor = north] at (3.96cm,-3.1cm)  {\small{$\underbrace{...........}_{\hs(k+1)}$}};
      \node[anchor = north] at (5.2cm,3.8cm)  {\tiny{$\hc(k+1)+1$}};
     \node[anchor = north,darkblue] at (-0.5cm,1.8cm)  {\tiny{$(i,\sigma'(i))$}};
     \node[anchor = north,lblue] at (2cm,0.85cm)  {\tiny{$(i,\sigma(i))$}};
\node[anchor = north] at (5.8cm,-3.15cm)  {\tiny{$(i,i)$}};
      \draw[dotted,line width=0.035cm] (0.805cm, 1.195cm) -- (5.7cm, -3.7cm);
\draw[dotted] (3.4cm, -1.4cm) -- (3.4cm, -2.9cm);
 \draw[darkblue,line width=0.04cm,decorate,decoration=snake,segment length=6pt,line cap=round] (0.805cm, 1.195cm) -- (-1cm, 3cm);
 \draw[lblue,line width=0.04cm,line cap=round] (0.805cm, 1.195cm) -- (2.68cm, -0.68cm);
  \draw[lblue,line width=0.04cm,line cap=round] (3.71cm, -0.68cm) -- (4.38cm, -1.38cm);
  \draw[lblue,line width=0.04cm,line cap=round] (2.71cm, -1.4cm) -- (3.75cm, -2.45cm); 
  \draw[lblue,line width=0.04cm,line cap=round] (4.9cm, -2.9cm) -- (4.4cm, -2.4cm); 
\end{tikzpicture}
\end{center}
The fact that each entry is acceptable can be proved as in the cases before.\\[1ex]
If $s$ is fixed and the assumption holds true for every $k\leq t$, then it also holds true for $t+1$ by an argumentation symmetric to the above one.\\[1ex]
 Therefore, we have found a permutation as wished for in every case.\\[1ex]
We can define the polynomials
 \[P_{k,l}\coloneqq \left\lbrace 
\begin{array}{ll}
 x^{n-a_k+\hd(i_{\min})-\vd(i_{\min})}, &\left. \begin{array}{l}
\textrm{if~there~is~a~ minimal~ element}~ i_{\min}~ \textrm{with} \\ 
\hb(i_{\min})=k~\textrm{and} ~\vb(\sigma(i_{\min}))=l; 
                                               \end{array}\right. 
\\  
0,  & ~~\textrm{otherwise}.
                              \end{array}\right.\]
Then, corresponding to the datum  $\Pa=((a_i)_{1\leq i\leq s},(a'_j)_{1\leq j\leq t},(P_{i,j})_{\begin{subarray}{l}
1\leq i\leq s \\ 
1\leq j\leq t \end{subarray}})$, there is an element $\mu\in K$, such that 
\[f(H)=\mu\cdot\prod\limits_{i=1}^r (H^{\Pa})_{i,\sigma(i)}\]
for every $H\in\Ha_U$.
\end{proof}
Given a toric invariant $f$, it thus suffices to find one acceptable permutation in order to calculate $f$ on $\N_U$.

\subsubsection{General description of toric invariants}\label{gendesctor}
We fix a sum-free toric invariant $f$ of block sizes $\underline{a}:=(a_1,\punkte,a_s)$ and $\underline{a}':=(a'_1,\punkte,a'_t)$ and assume, without loss of generality, $a_1\leq \punkte\leq a_s$ and $a'_1\leq \punkte\leq a'_t$.\\[1ex]
Given an integer $i\in\{1,\punkte,s\}$, we define $s_i:=\sum\limits_{l=1}^i a_l +1$.
\begin{lemma}\label{accperm}
The permutation $\sigma\in S_r$ defined by
\[\sigma(i)\coloneqq\left\lbrace 
\begin{array}{ll}
i+\hs(i_{k}), &\textrm{if}~\sum\limits_{j=1}^{j_{k}-1} a'_j <	i \leq\sum\limits_{j=1}^{i_{k}} a_j;   \\ 
i-\ch(i_{k}), &\textrm{if}~\sum\limits_{j=1}^{i_{k}} a_j < i \leq \sum\limits_{j=1}^{j_{k}} a'_j;  \\ 
i, & \textrm{otherwise}. 
                         \end{array}\right.\]
for $k\in\{1,\punkte, x\}$  is acceptable for $(\underline{a},\underline{a}')$.
\end{lemma}
\begin{proof}
The proof is given by a straight forward calculation making use of the fact that for \[i\notin\{s_{i_k}+h\mid h\in\{0,\punkte,\hs(i_k)\} ~\textrm{and} ~k\in\{1,\punkte,x\}\},\]
the entry $(i,i)$ is acceptable for  $(\underline{a},\underline{a}')$.
\end{proof} 
We fix the acceptable permutation $\sigma\in S_r$ and the induced datum $P$ (as in the proof of Lemma \ref{secred}).
\begin{lemma}
 Let  $f$ be a sum-free toric invariant of block sizes $\underline{a}\coloneqq(a_1,\punkte,a_s)$ and $\underline{a}'\coloneqq(a'_1,\punkte,a'_t)$ and let $f(H)= x_1^{h_1} \punkte x_{n-1}^{h_{n-1}} $. Then $h_{n-1}=s$ and for $l\in \{1,\punkte,n-2\}$:
 \[h_l= t+\sum_{k=2}^l \sharp\{j\in\{1,\punkte,t\}\mid a'_j\geq k\} -\sum_{k=1}^{l-1} \sharp\{i\in\{1,\punkte,s\}\mid a_i\geq n-k\}\]
\end{lemma}
\begin{proof}
Let $H=H_{\tor}\in \Ha_U$ be a matrix with entries $H_{k+1,k}=:x_k$. Then \[H^{P}_{(i,\sigma(i))}=(H^{n-\vd(\sigma(i))-a_{\hb(i)}+\hd(i)})_{(\hd(i),\vd(\sigma(i))}\]
\[H^{P}_{(i,\sigma(i))}=(H^{n-\vd(\sigma(i))-a_{\hb(i)}+\hd(i)})_{\hd(i),\vd(\sigma(i))}=\prod\limits_{k=\vd(\sigma(i))}^{n-a_{\hb(i)}+\hd(i)-1}x_k.\] 
The proof follows from combinatorial considerations, then.
\end{proof}

\subsection{The associated toric variety}
We  denote the subring of $K[\N]^U$ which is generated by all toric invariants by $K[\N]^U_{\tor}$. 
Corresponding to $K[\N]^U_{\tor}$, there is a variety $X\coloneqq \Spec K[\N]^U_{\tor}$ which is a toric variety.
Given a sum-free toric invariant, there are integers $h_1,\punkte,h_{n-1}$, such that \[f(H)=x_1^{h_1}\cdot \punkte\cdot x_{n-1}^{h_{n-1}}.\]
Denote by $S$ the set of tuples $(h_1,\punkte,h_{n-1})\in \textbf{N}^{n-1}$ that arise in this way from a minimal set of generating toric invariants and denote $\sigma\coloneqq \Cone(S)$.\\[1ex]
Let $N$ be the lattice $\textbf{Z}^{n-1}$, then $\sigma$ is generated by the finite set $S\subset \textbf{Z}^{n-1}$ and fulfills $\sigma \cap (-\sigma)=\{0\}$, therefore, $\sigma$ as well as $\sigma^\vee$ are strongly convex rational polyhedral cones of maximal dimension. The variety $X= \Spec K[\N]^U_{\tor}\cong\Spec K[S_{\sigma^{\vee}}]$, thus, is a normal toric variety by Lemma \ref{toricco}.\\[1ex]
Let $T\subset \GL_n$ be the torus of diagonal matrices. There is a natural action $\tau$ of $T$ on the $U$-invariant ring of $\N$ as follows: 
\[\tau\colon ~ T\times K[\N]^U \rightarrow K[\N]^U;~ (t,f)~  \mapsto \left( \begin{array}{ll}
f\colon & \N \rightarrow K \\ 
 & N\mapsto f(tNt^{-1})
                 \end{array}\right).\]
Another operation is given, since the variety $X=\Spec K[\N]^U_{\tor}$ is a toric variety:
\begin{align}
\tau'\colon &~ (K^*)^{n-1}\times K[\N]^U_{\tor} \rightarrow K[\N]^U_{\tor}.\nonumber
\end{align} 
Let $f$ be a toric invariant, such that $f(H)=x_1^{h_1}\punkte x_{n-1}^{h_{n-1}}$, and let $c\coloneqq(c_1,\punkte,c_{n-1})\in(K^*)^{n-1}$. Then
$\tau'(c,f)(H)= f(H)\cdot c_1^{h_1}\punkte c_{n-1}^{h_{n-1}}.$\\[1ex]
The  operation $\tau$ is induced by the operation $\tau'$ via the morphism
\[\rho\colon ~  T \rightarrow (K^*)^{n-1};~ (t_1,\punkte,t_n)\mapsto (t_{2}/t_1,\punkte,t_{n}/t_{n-1}).\]

Let $i\in\{1,\punkte,n-1\}$, then we define the $U$-invariant $\det_{i}(N):=\Det(N^{n-i}_{(i,i)})$ and the $U$-invariant $f_i$ to be the unique toric invariant of block sizes $(i),(1,\punkte,1)$. 
 Furthermore, for integers $i,j\in\{1,\punkte,n\}$, such that $j< i-1$, we define the datum \[P=\left((j-1,n-i+1),(j,n-i),\left(\begin{array}{cc}
x^{n-j+1} & 0 \\ 
x & x^i
                         \end{array} \right)\right)\] 
and denote $f_{i,j}\coloneqq f^{P}$. These invariants separate the $U$-orbits generically in $\N_U\subseteq \N$.\\[1ex]
Let $\pi: \N \rightarrow \N\quot U$ be an algebraic $U$-quotient of $\N$ which exists, since $K[\N]^U$ is finitely generated. The variety $\N\quot U$ is normal, since the nilpotent cone is normal (see \cite[III.3.3]{Kr}).\\[1ex]
The space of $U$-normal forms is given by $\Ha_U\cong \textbf{A}^D\times (K^*)^{n-1}$ and the map $\pi$ restricts to a morphism $i:\Ha_U\rightarrow \N\quot U$. We consider the toric variety $X$ described above by its cone $\sigma$ which is induced by the sum-free toric invariants and let $X'\cong (K^*)^{n-1}$ be the dense orbit in $X$.\\[1ex]
The morphism $i: \Ha_U\rightarrow i(\Ha_U)$ is injective, since   the fibres are separated generically by certain $U$-invariants. Therefore, we can construct an explicit morphism $i': i(\Ha_U)\rightarrow \Ha_U$, such that $i\circ i'=\id_{ i(\Ha_U)}$ and $i'\circ i=\id_{ \Ha_U}$. The morphism $i$ is, thus, birational  and $\textbf{A}^D\times (K^*)^{n-1}\cong i(\Ha_U)\subseteq \N\quot U$. 
\begin{lemma}
 The natural embedding $K[\N]^U_{\tor}\rightarrow K[\N]^U$ induces a dominant, $T$-equivariant morphism $p: \N\quot U\rightarrow X,$  such that $p^{-1}(x)\cong \textbf{A}^D$ for each point $x'\in X'$. 
\end{lemma}
\begin{proof}
The morphism $p$ is clearly dominant and $T$-equivariant due to our considerations above. \\[1ex]
Let $x'\in X'$, then $p^{-1}(x)\subseteq i(\Ha_U)$, since every determinant ${\det}_i$ for $i\in\{1,\punkte,n-1\}$ is a toric invariant. If $x'\in X'$, none of these determinants vanishes on $x'$ and Section \ref{gnfsect}, therefore, yields $p^{-1}(x')\subseteq i(\Ha_U)$. Since the orbits in $\N_U$ are separated by certain $U$-invariants  and since $\Ha_U\cong \textbf{A}^D\times X'$, the claim  $p^{-1}(x)\cong \textbf{A}^D$ follows.
\end{proof}
 There is a morphism $q: \N\quot U\rightarrow \textbf{A}^D$ as well, such that the composition
\[\Ha_U\xrightarrow{i} \N\quot U\xrightarrow{q} \textbf{A}^D\]
yields $q\circ i(H)= (x_{i,j})_{1< j+1\leq i-1< n}\in \textbf{A}^D$.

\begin{lemma}
 The morphism
\[(q,p): \N\quot U\rightarrow \textbf{A}^D\times X\]
is dominant and birational.
\end{lemma}
\begin{proof}
 The morphism $(p,q)$ is dominant, since  $\textbf{A}^D\times X'\subseteq \im(p,q)\subseteq \textbf{A}^D\times X$.\\[1ex]
The morphism $(p,q)$ is birational, since $(p,q)$ is dominant and generically injective: the fibre $(p,q)^{-1}(y)$ contains exactly one element for every $y\in \textbf{A}^D\times X'$, since the $U$-orbits can be separated in $\textbf{A}\times X'$. More straight forward, $(p,q)$ restricts to an isomorphism $i(\Ha_U)\cong \textbf{A}^D\times X'$.
\end{proof}
Note that the morphism $(p,q)$ is not surjective for $n\geq 4$. Even in the case $n=4$, we can show $K[\N]^U\ncong K[\textbf{A}^3]\otimes K[\N]^U_{\tor}$ and $\N\quot U\ncong \textbf{A}^3\times X$.\\[1ex]
We define a $U$-invariant $g$ by the data 
\[P=\left\lbrace 
\begin{array}{ll}
((2),(2), (x)), &~\textrm{if}~n=4; \\ 
((n-2),(2,n-4), (x,x^4)) &~\textrm{otherwise}.
\end{array}
\right. \]   
Then 
$g(H)=(x_{3,1}\cdot x_{4,2}-x_2\cdot x_{4,1})\cdot \Det_{n-4}(H)$
and the relation 
\[g\cdot \underbrace{\Det_{n-3}\cdot\Det_1\cdot f_{n-3}\cdot f_{n-1}}_{\coloneqq F} =\underbrace{f_{3,1}\cdot f_{4,2}\cdot f_{n-3} \cdot f_{n-1} - f_{4,1}\cdot f_{n-2}^2 \cdot \Det_{n-3}\cdot \Det_1}_{\coloneqq F'}\]
holds true in $K[\N]^U$. The set 
$M\coloneqq \{\underline{x}\in\textbf{A}^D\times X\mid F(\underline{x})= 0; F'(\underline{x})\neq 0\}$ is non-empty and 
 the inclusion $M \subseteq (\textbf{A}^D\times X)\backslash \im(p,q)$ directly yields that the morphism $(p,q)$ is not surjective.

\section{Towards a GIT-quotient for the Borel-action}\label{Borelquot} 
We initiate the study of a GIT-quotient for the Borel action on  $\N$ and start by discussing $n=2$.
\begin{example}\label{Btwo}
 Example \ref{Utwo} proves $K[\N]^U=K[f_{2,1}].$
The $U$-invariant morphism $f_{2,1}$ is a $B$-semi-invariant of weight $\chi_{0}\coloneqq\omega_2 -\omega_1$. Therefore,
\[\bigoplus_{\chi\in X(B)}\bigoplus_{n\geq 0}K[\N]^{B,n\chi}~=~\bigoplus_{n\geq 0}K[\N]^{B,n\chi_{0}}.\]
Of course, $N\in\N_B$ if and only if $f_{2,1}(N)\neq 0$ and therefore $\N^{\chi_{0}-\sst}=\N_B$.\\[1ex]
The morphism 
 \begin{align}
\mu\colon& ~\N^{\chi_{0}-\sst}\rightarrow \Proj K[f_{2,1}]=\{1\}; ~~~N \mapsto f_{2,1}(N)=1,  \nonumber
\end{align}
thus, is a GIT-quotient.
\end{example}

\begin{example}\label{Bthree}
Let us consider $n=3$.
Example \ref{Uthree} proves \[K[\N]^U=K[f_{3,1},f_{1},f_{2},\Det_1]/\left( f_1\cdot f_2= \Det_1^3\right).\]
We consider these $U$-invariants:
\begin{enumerate}
 \item  $f_{3,1}$ and $\Det_1$ are $B$-semi-invariants of weight $\chi_{3,1}\coloneqq\omega_3 -\omega_1$,
\item  $f_1$ is a $B$-semi-invariant of weight $\chi_{1}\coloneqq -2\omega_1+\omega_2+\omega_3$ and
\item  $f_2$ is a $B$-semi-invariant of weight $\chi_{2}\coloneqq -\omega_1-\omega_2+2\omega_3$.
\end{enumerate}
The equality $\Det_1=\Det_2$ holds true on $\N$, therefore $\N^{\chi_{3,1}-\sst}=\N_B\cup \{N\in\N\mid N_{3,1}\neq 0\}$.\\[1ex]
Thus, the morphism 
 \begin{align}
\mu\colon& ~\N^{\chi_{3,1}-\sst}\rightarrow \mathbf{P}^1=\Proj K[f_{3,1},\Det_1];~~~N \mapsto (f_{3,1}(N):\Det_1(N)) \nonumber
\end{align}
is a GIT-quotient.
\end{example}

\subsection{Generic separation of the same weight}\label{Bgensep}
We define the character
\[\chi\coloneqq\sum_{i=1}^{n-1}(\omega_{n-i+1}+\ldots+\omega_n) - \sum_{i=1}^{n-1}(\omega_1+\ldots+\omega_i).\]
and show how to extract the entries of the normal forms $H$ in the affine space $\Ha_B\cong\textbf{A}^D$ of dimension $D\coloneqq\frac{(n-1)(n-2)}{2}$ with the generating semi-invariants from Proposition \ref{semiprop}. In particular, we are able to separate them with semi-invariants of the same weight $\chi$.
\begin{lemma}\label{extract}
For each $i$ and $j$, such that $2<j+2\leq i\leq n$, there is a semi-invariant $g_{i,j}$ of weight $\chi$ which fulfills \[g_{i,j}(H)=H_{i,j}\]
for every normal form $H\in\Ha_B$.
\end{lemma}
\begin{proof}

Let $n-i+1 \notin\{j-1, j\}$ and
define the datum $\Pa\coloneqq((a_k)_k,(a'_k)_k,(P_{k,l})_{k,l})$ by
\begin{itemize}
 \item[$\cdot$] $(a_k)_{1\leq k\leq n-1}\coloneqq(j-1,n-i+1,j,1,\ldots,j-2,j+1,\ldots,n-i,n-i+2,\ldots,n-1)$,
\item[$\cdot$] $(a'_k)_{1\leq k\leq n-1}\coloneqq(j,n-i+1,j-1,1,\ldots,j-2,j+1,\ldots,n-i,n-i+2,\ldots, n-1)$,
\item[$\cdot$] $P_{k,l}\coloneqq\left\lbrace 
\begin{array}{ll}
x^{n-j+1}, &~\textrm{if}~k=l\in\{1,3\}; \\ 
x & ~\textrm{if}~k=2~\textrm{and}~l=1; \\ 
x^i & ~\textrm{if}~k=l=2; \\ 
x^{i-j} & ~\textrm{if}~k=3~\textrm{and}~l=2;\\
x^{n-a_k} & ~\textrm{if}~k=l>3;\\
0 & ~\textrm{otherwise}.
                        \end{array}\right.$
\end{itemize}
Let us denote $g_{i,j}\coloneqq f^{\Pa}$ and let $H\in\Ha_B$, then\\[1ex]
$g_{i,j}(H) =  \det(H^{\Pa}) = \det((P_{k,l}(H)_{(a_k,a'_l)})_{1\leq k,l\leq 3}) \cdot \det((P_{k,l}(H)_{(a_k,a'_l)})_{4\leq k,l\leq n-1}) =$\\[1ex]
$=  \det((P_{k,l}(H)_{(a_k,a'_l)})_{1\leq k,l\leq 3}) \cdot \prod_{k=4}^{n-1}\det(P_{k,k}(H)_{(a_k,a'_k)}) = \det((P_{k,l}(H)_{(a_k,a'_l)})_{1\leq k,l\leq 3}) =$\\[1ex]
$ = \det \left(\begin{array}{l|l|l}
\begin{array}{cccc}1&&0&0\\ &\ddots&&\vdots\\ *&&1&0\end{array} &~~~~~~~~ 0 &~~~~~~~ 0 \\ \hline
~~H_{(n-i+1,j)} & \begin{array}{cccc}0&&&0\\ 1&&&\\ &\ddots&&\\ *&&1&0   \end{array} & ~~~~~~~0 \\ \hline
~~~~~~~~0 & \begin{array}{cccc}*&\cdots&*&1\\ *&\cdots&*&*\\ \vdots&&\vdots&\vdots \\  *&\cdots&*&* \end{array} & \begin{array}{ccc}0&\cdots&0\\   1&&0\\ &\ddots&\\ *&&1\end{array}
          \end{array}\right)  = H_{i,j}.$\\[1ex]
In the remaining cases, the argumentation is the same as in this first case:\\[1ex]
If $n-i+1=j$, then we define the datum $\Pa\coloneqq((a_k)_k,(a'_k)_k,(P_{k,l})_{k,l})$ by
\begin{itemize}
 \item[$\cdot$] $(a_k)_{1\leq k\leq n-1}\coloneqq(j-1,j,1,\ldots,j-2,j+1,\ldots, n-1)$,
\item[$\cdot$] $(a'_k)_{1\leq k\leq n-1}\coloneqq(j,j-1,1,\ldots,j-2,j+1,\ldots, n-1)$,
\item[$\cdot$] $P_{k,l}\coloneqq\left\lbrace 
\begin{array}{ll}
x^{n-j+1}, &~\textrm{if}~k=l\in\{1,2\}; \\ 
x & ~\textrm{if}~k=2~\textrm{and}~l=1; \\  
x^{n-a_k} & ~\textrm{if}~k=l>2;\\
0 & ~\textrm{otherwise}.
                        \end{array}\right.$
\end{itemize}
Let $n-i+1=j-1$ and $j=2$, that is, $i=n$. We define the datum $\Pa=((a_k)_k,(a'_k)_k,(P_{k,l})_{k,l})$ as follows:
\begin{itemize}
 \item[$\cdot$] $(a_k)_{1\leq k\leq n-1}\coloneqq(2,1,3,\ldots,n-1)$,
\item[$\cdot$] $(a'_k)_{1\leq k\leq n-1}\coloneqq(1,2,3,\ldots,n-1)$,
\item[$\cdot$] $P_{k,l}\coloneqq\left\lbrace 
\begin{array}{ll}
x^{n-2}, &~\textrm{if}~k=l=1; \\ 
x^{n-1} & ~\textrm{if}~k=1~\textrm{and}~l=2; \\
x & ~\textrm{if}~k=l=2; \\  
x^{n-k} & ~\textrm{if}~k=l>2;\\
0 & ~\textrm{otherwise}.
                        \end{array}\right.$
\end{itemize}
Let us assume $n-i+1=j-1$ and $j\geq 3$ and consider the datum $\Pa=((a_k)_k,(a'_k)_k,(P_{k,l})_{k,l})$ defined by
\begin{itemize}
 \item[$\cdot$] $(a_k)_{1\leq k\leq n-1}\coloneqq(j,j-1,1,\ldots,j-2,j+1,\ldots, n-1)$,
\item[$\cdot$] $(a'_k)_{1\leq k\leq n-1}\coloneqq(1,j,j-1,2\ldots,j-2,j+1,\ldots, n-1)$,
\item[$\cdot$] $P_{k,l}\coloneqq\left\lbrace 
\begin{array}{ll}
x^{n-j+1} & ~\textrm{if}~(k=1~\textrm{and}~l=2)~\textrm{or~if}~k=l=3; \\ 
x & ~\textrm{if}~k=l=2; \\ 
x^{n-j+2} & ~\textrm{if}~k=2~\textrm{and}~l=3;\\
x^{n-a_k} & ~\textrm{if}~k=l=1~\textrm{or~if}~k=l>3;\\
0 & ~\textrm{otherwise}.
                        \end{array}\right.$
\end{itemize}
It follows from Proposition  \ref{semiprop} that every such semi-invariant $g_{i,j}$ is of weight $\chi$.\qedhere
\end{proof}
We have, thus, found semi-invariants of the same character that extract the coordinates of $\Ha_B\cong \textbf{A}^{D}$. 
As the translation to the representation theory of the algebra $K\Q/I$ provides an insight into the classification of finite parabolic actions in case the algebra is representation-finite (see \cite{BoRe}), the translation to the language of moduli spaces may provide further knowledge about quotients if the algebra is of wild representation type.

\end{document}